\documentclass{mcom-l}
\usepackage{amsmath}
\usepackage{amsfonts}
\usepackage{amssymb}
\usepackage{amsthm}
\usepackage{graphicx}
\usepackage{cite}
\usepackage{url}

\usepackage{cleveref}
\usepackage{latexsym}

\usepackage{mathtools}
\usepackage{listings}
\newtheorem{theorem}{Theorem}[section]
\newtheorem{proposition}[theorem]{Proposition}

\newtheorem{lemma}[theorem]{Lemma}
\newtheorem{corollary}[theorem]{Corollary}
\newtheorem{remark}[theorem]{Remark}

\theoremstyle{remark}

\newcommand{\spn}{\mathop{\mathrm{span}}}

\newcommand{\nats}{\mathbb{N}}
 
\newcommand{\RR}{\mathbb{R}}

\newcommand{\M}{\mathbb{M}}
\newcommand{\calo}{\mathcal{O}}

\newcommand{\sph}{\mathbb{S}} 

\newcommand{\bfa}{\mathbf{a}}

\renewcommand{\d}{\mathrm{dist}}

\newcommand{\bfbeta}{\boldsymbol \beta}
\newcommand{\cF}{{\mathcal F}}

\newcommand{\dist}{\mathrm {dist}}

\def\calh{{\mathcal H}}

\def\caln{{\mathcal N}}

\newcommand{\set}[1]{\{ #1 \}}
\newcommand{\dotp}[2]{\langle #1 , #2 \rangle}

\newcommand{\proj}{\mathsf P}

\numberwithin{equation}{section}

\title[A Novel Meshless Galerkin Method]{A Novel Galerkin Method for Solving PDEs on the Sphere Using
  Highly Localized Kernel Bases} \author{Francis J.~Narcowich}
\address{ Department of Mathematics, Texas A\&M University, College
  Station, TX 77843, USA. } \email{fnarc@math.tamu.edu}
\thanks{Research supported by grant DMS-1211566 from the National
  Science Foundation.}
    
\author{Stephen T. Rowe}
\address{ Department of Mathematics, Texas A\&M
    University, College Station, TX 77843, USA.}
\email{srowe@math.tamu.edu}  
\thanks{Research supported by grant DMS-1211566 from the National
    Science Foundation and Sandia National Laboratories}

\author{Joseph D.~Ward}
\address{ Department of Mathematics, Texas A\&M
    University, College Station, TX 77843, USA. }
\email{jward@math.tamu.edu}
\thanks{Research
    supported by grant DMS-1211566 from the National
    Science Foundation.}



\begin{document}

\begin{abstract}
The main goal of this paper is to introduce a novel meshless kernel 
Galerkin method for numerically solving partial differential equations 
on the sphere. Specifically, we will use this method to treat the 
partial differential equation for stationary heat conduction on 
$\mathbb S^2$, in an inhomogeneous, anisotropic medium. The Galerkin method used to do this employs spatially well-localized, ``small footprint'',  robust bases for the associated kernel space. The stiffness matrices arising in the problem have entries decaying 
exponentially fast away from the diagonal. Discretization is achieved by first zeroing out small entries, resulting in a sparse matrix, and then replacing the remaining entries by ones  computed via a very efficient kernel quadrature formula for the sphere. Error estimates for the approximate  Galerkin solution are also obtained. 
\end{abstract}

\maketitle

\section{Introduction}

The main goal of this paper is to introduce and analyze a novel meshless kernel Galerkin method for numerically solving partial differential equations on the sphere. Specifically, we will use this method to treat the partial differential equation for stationary heat conduction on $\sph^2$, the unit sphere in $\RR^3$,  in an inhomogeneous, anisotropic medium. The equation for this heat-flow is
\begin{equation}
\label{pde_dyadic}
Lu  = -\text{\rm div}( \mathbf a \! \cdot \! \! \nabla u) + b(x)u=f,
\end{equation}
where div and $\nabla$ are the divergence and gradient on $\sph^2$, and $\bfa$ is a rank 2 positive definite tensor on $\sph^2$, and $f$ is in the Sobolev space $H_s$, $s\ge 0$. The analysis includes error estimates when the exact stiffness matrix is used, and also when various quadrature-based discretizations of that matrix are employed.

The kernels that we employ here are surface splines, $\phi_m(x\cdot
y)= (-1)^m(1-x\cdot y)^{m-1}\log(1-x\cdot y)$, $m\ge 2$. These are
conditionally positive definite spherical basis functions (SBF). Their
reproducing kernel Hilbert spaces (native spaces) are equivalent to
the Sobolev spaces $H_m\approx W_2^m(\sph^2)$.  The associated
approximation spaces involve spans of \{$\phi_m((\cdot)\cdot
\xi)\}_{\xi \in X}$, with $X$ being a discrete, finite set of
quasi-uniformly distributed centers or nodes, along with spherical
harmonics of order $m$. We denote them by $V_{\phi_m,X}$.

It is well known that, under mild conditions on an SBF $\phi$, the
spaces $V_{\phi,X}$ have excellent approximation power
\cite{mhaskar-etal-2010}. This makes them an an obvious choice for use
in meshless methods for solving PDEs. Kernel Galerkin methods using
radial basis functions (RBFs) were employed in \cite{Wendland-99-2} to
theoretically treat elliptic partial differential equations on
$\RR^n$.  On $\sph^2$, SBF Galerkin methods for $-\Delta u = f$ have
been studied in \cite{LeGia-2004, LeGia-2005}, More recently, Le Gia
\emph{et al.} \cite{LeGia-et-al-10-1, LeGia-et-al-11-1} used
collocation and multi-level SBF methods for the purpose.

There are several drawbacks to these methods.  Bases of the form
$\{\phi((\cdot)\cdot \xi)\colon \xi\in X\}$ give rise to
interpolation/stiffness matrices that are full and
poorly-conditioned. The bases are not well localized spatially:
changing even small amounts of data requires re-computation of the
matrices involved \cite[pg.~208]{wendland_book}. For Galerkin methods
on $\sph^2$, there is an additional problem that arises in connection
with discretization. Entries in the stiffness matrices have to be
numerically computed via quadrature. For instance, the method used in
\cite{LeGia-2004} requires solving an optimization problem to find the
weights involved. This is a computationally expensive
process. Applying it to large numbers of nodes is problematic.

The new meshless Galerkin method that we present here overcomes these
difficulties. There are two novel features of our Galerkin approach to
numerically approximating solutions to \eqref{pde_dyadic}. First, for
the surface splines, Fuselier \emph{et al.}  \cite{FHNWW2012} recently
showed in the Lagrange basis for $V_{\phi_m,X}$ each Lagrange function
$\chi_\xi(x)$ is highly localized spatially; indeed, $\chi_\xi(x)$
decays exponentially fast as $x$ moves away from $\xi$. Moreover, when
$\chi_\xi$ is expressed in the kernel basis for $V_{\phi_m,X}$, the
coefficient $\alpha_{\xi,\eta}$ of $\phi_m(x\cdot \eta)$ also decays
exponentially fast as $\eta$ moves away from $\xi$.  Thus $\chi_\xi$
has a ``small footprint'' in the kernel basis. These features make the
Lagrange basis robust. It is for this basis that we will build our initial
theory. 

$L_2$-error estimates for the Galekin solutions constructed using the thin-plate splines are derived in section~\ref{SBF_tau_est}. Indeed, we derive the error estimates for Galerkin solutions constructed using an arbitrary SBF $\phi$, restricted only
by the condition that coefficients in its expansion in spherical
harmonics satisfy \eqref{tau_assumpt}. (These conditions hold for
$\phi_m$, with $\tau=m$.) If $f\in H_s(\sph^2)$, the $L^2$ error bounds
derived in section~\ref{SBF_tau_est} are $C\|f\|_s h_X^{s+2}$, for
$0\le s \le 2m-2$, and $C\|f\|_{2m-2}h_X^{2m}$, for $s>2m-2$. These
error estimates differ from those found in works cited above in two
ways: they hold for operators of the form $L$ in \eqref{pde_dyadic},
not just $\Delta$, and they apply even when $s$ is fractional.

Constructing  the $\chi_\xi$'s requires all of the points in $X$. In this sense, the $\chi_\xi$'s form  a \emph{global} Lagrange basis; finding them is computationally expensive. In \cite{FHNWW2012} a \emph{local} Lagrange basis was
also introduced. Each basis element $
\chi_\xi^{loc}$ is constructed using only $\calo((\log N))^2)$ in the
neighborhood of $\xi$, and it approximates $\chi_\xi$ well. These
elements have very fast spatial decay, although not exponential. They
have the advantage that computing them is fast and
parallelizable. 

The computational advantages of these local bases make them a good choice for the implementation of our meshless method. We use our initial theory for the global Lagrange basis to further develop this method when a local Lagrange basis is used.

The stiffness matrix matrix $A$ in the global Lagrange basis is the key to the whole method. Our technique relies on having high quality approximations to this matrix. $A$ itself has a number of
very attractive properties in the global basis: (1) The entries $A_{\xi,\eta}$ \emph{decay
exponentialy fast} in the distance between $\xi$ and $\eta$. Thus $A$
is essentially sparse and, as we shall see, zeroing out small entries will provide a sparse approximation.
(2)~If $N_X=\text{card}(X)$, then the number of \emph{non-negligible}
entries in each row is $\calo\big((\log(N_X)^2\big)$. (3) The
\emph{condition number} of $A$ is $\calo(q_X^{-2})$, where $q_X$ is
half the separation distance for $X$.

Discretization of the stiffness matrix is essential to the Galerkin method. After zeroing out small entries, the remaining matrix is sparse, having $\calo\big(N_X(\log(N_X)^2\big)$ entries. These entries are integrals that have to be computed via quadrature. To do
this, we use kernel-based quadrature formulas for the
sphere\cite{sommariva_womer2005, hesse-et-al-2010, FHNWW2013}, with
our kernels being surface splines. The special bases available to us
for the approximation spaces corresponding to the $\phi_m$'s enable us
to efficiently construct the quadrature
formulas\cite{FHNWW2013}. Unlike the quadrature formula used in
\cite{LeGia-2004}, the weights are obtained by solving a \emph{linear}
system of equations. Finding the weights follows readily by applying
the techniques of \cite{FHNWW2012}. These quadrature formulas are
accurate, optimally so in many cases, even in the presence of noise,
and they are stable when the number of nodes, which we denote by $Y$, increases. Indeed, tests
run with this method handled over two-thirds of a million nodes
\cite{FHNWW2013}.

When the problem is discretized, we get the continuous error plus a
quadrature error. The quadrature error comes from replacing the stiffness matrix $A$ by $A^Y$, whose entries are computed from those of $A$ via the Theorem~\ref{error_l2_norm_A} gives a theoretical
bound on $\|A-A^Y\|_2 $. If we ignore logarithms and similar terms,
this is $\|A-A^Y\|_2 \sim (N_X/N_Y)^{M}$, where $\phi_M$, $2\le M \le
m$, is used to obtain weights in the quadrature formula.  The $L^2$
error estimate for the discretized Galerkin solution is given in
Theorem~\ref{discrete_galerkin_error}.

There is related work for $\RR^n$. As mentioned earlier, Wendland
\cite{Wendland-99-2} explored RBF Galerkin for domains in
$\RR^n$. However, problems with discretizing the stiffness matrix
entries were limiting factors in implementing the method. For
$\sph^n$, the quadrature methods discussed above avoid these
problems. In \cite{BondLehoucqRowe}, Bond \emph{et al.} successfully
employed an RBF Galerkin method, using an approach based on the one
used here, in numerical experiments for a peridynamic model of a
nonlocal heat equation (See \cite{QiangDuSIAMNews} for a discussion of
peridynamics.)

We now discuss an outline of the paper and its organization. In
section~\ref{preliminaries}, we discuss background information on
quasi-uniform sets of centers, geometry, and Sobolev spaces.  In
section~\ref{SBFs}, we discuss the relevant approximation spaces for
the Galerkin method that we
introduce. Section~\ref{small_footprint_bases} describes the major
tools used here: the highly localized Lagrange bases we employ
throughout the work (section~\ref{lagrange_functions}) and the
quadrature formulas essential for discretizing the stiffness matrix
(section~\ref{quadrature_formulas}). Section~\ref{Lu=f_prop} discusses
general properties of \eqref{pde_dyadic}, including regularity of weak
solutions and a useful application of the ``Nitsche trick"
\cite{Nitsche-71}. Section~\ref{SBF_tau_est} gives Galerkin error
estimates in the case where the approximation spaces are generated by
SBFs satisfying only \eqref{tau_assumpt}. The next section,
section~\ref{A_lagrange_basis}, is key. It sets forth the properties
of the stiffness matrix in the Lagrange basis: its quasi-bandedness,
good conditioning, decay of matrix elements away from the diagonal,
and general robustness.  Section~\ref{discretized_galerkin_solution}
discusses aspects of discretizing the problem. Theorem
\ref{discrete_galerkin_error} is the main result of the section; it
contains the theoretical $L^2$-bound on the difference between
discretized Galerkin solution and the exact solution.  In section~\ref{sparse_app_local_lag}, we  discusse reducing the computational expense of numerically finding the Galerkin solution to the problem. There are two aspects of this. The first is obtaining a truncated approximation $\widetilde A^Y$ to the discretized stiffness matrix $A^Y$. Each row in $\widetilde A^Y$ has $\calo(\log(N_X))^2)$ nonzero entries, as opposed to $N_X$ in $A^Y$. The second is to replace the global Lagrange basis with a local one, which is much easier to numerically find. The error estimates from making these approximations are virtually unchanged. In
section~\ref{numerics}, the results of numerical experiments that we
did are presented. In terms of rates of convergence, the numerical
results were actually better than the theory predicted.  Finally, we
wish to mention a few new results discussed in the appendix.  In
section~\ref{SBF_approximation_power}, we establish a generalized
version of the so called ``doubling trick,'' which applies to SBF
interpolation of functions twice as smooth as those in the native
space of the SBF \cite{Schaback-00-1, Fuselier_Wright_2012}.  The
result, which is given in Theorem~\ref{cond_pos_def_case}, holds for
functions smoother that those in the native space for $\phi$, but not
having ``double'' the smoothness. In addition, it applies for SBFs
that are conditionally positive definite.

\section{Preliminaries}
\label{preliminaries}


\subsection{Geometry of Sets of Centers}\label{sets_centers}

Although $\sph^2$ is the underlying space for the Galerkin methods
treated in this paper, much of what we will discuss in the next few
sections applies to $\sph^n$. In view of this, we will work in
$\sph^n$.

Let $\dist(x,y)$ be the geodesic (great circle) distance between two
points on $\sph^n$. We will let $X = \{x_j\}_{j=1}^N \subset \sph^n$
be a set of $N$ distinct points; we will call $X$ a set of centers. We
remark that, apart from, say the vertices of the Platonic solids in
$\sph^2$ and similar quantities in $\sph^n$, $n>2$, there are no
uniformly distributed sets of points in $\sph^n$. We can, however,
obtain \emph{quasi-uniform} sets of points. We will explain this
below.

There are three geometrical quantities associated with $X$. The first
is the \emph{separation radius}, $q_X$.  For every $x\in X$, the
radius of the largest ball whose interior contains no other point of
$X$ is given by $\frac{1}{2}\dist(x,X\setminus \{x\})$; $q_X$ is
defined to be the smallest of these radii:
\[ q_X:=\min_{x\in X} \frac{1}{2}\dist(x,X\setminus \{x\}).
\] Obviously $2q_X$ is the minimum distance between any two points in
$X$. The second and third are the \emph{mesh norm}, $h_X$, and the
mesh ratio, $\rho_X$. The mesh norm $h_X$ is the radius of the largest
ball in $\sph^n$ whose interior contains no point of $X$. It also can
be characterized as the largest distance of any point in $\sph^n$ from
$X$. The mesh ratio is the ratio of $h_X$ to $q_X$:
\[ h_X:= \max_{x\in \sph^n}\dist (x,X) \ \mbox{and } \rho_X:=
h_X/q_X.
\] The mesh norm, which is also called the \emph{fill distance},
measures how tightly packed the centers are in $\sph^n$. The mesh
ratio measures how uniformly the centers are placed. When it is close
to 1, the distribution of the points in $X$ is said to be \emph{quasi
uniform}.

For $\rho\ge 1$, let $\cF_\rho = \cF_\rho(\sph^n)$ be the family of
all sets of centers $X$ with $\rho_X\le \rho\,$; we will say that the
family $\cF_\rho$ is \emph{$\rho$-uniform}.  Unless confusion would
arise, we will not indicate either $\rho$ or $\sph^n$, and just use
$\cF$ to designate the family $\cF_\rho(\sph^n)$.  The specific $\rho$
or sphere $\sph^n$ will be clear from the context.

On $\sph^2$, there are three important quasi-uniform sets of nodes
(centers): Fibonaccii nodes, icosahedral nodes, quasi minimum energy
nodes. All three of these families of nodes are quite popular in
applications; see, for
example~\cite{Giraldo:1997,StuhnePeltier:1999,Ringler:2000GeodesicGrids,Majewski:2002GME}
for the icosahedral
nodes,~\cite{SwinbankPurser:2006,SlobbeSimonsKlees:2012,HuttigKai:2008}
for the Fibonacci nodes,
and~\cite{WrightFlyerYuen,flyer_wright2009,FlyerLehtoBlaiseWrightStCyr2012,SWFK2012}
for the quasi-minimum energy nodes. Similar considerations apply
to $\sph^n$, $n>2$.

\subsection{Spherical Harmonics and Sobolev Spaces}

The sphere $\sph^n$ is of course a Riemannian manifold with metric
tensor $g_{ij}$ and invariant measure $d\mu = \sqrt{\det(g_{ij})}
dx^1\cdots dx^n$, where $x^1,x^2, \ldots, x^n$ is a smooth set of
local coordinates. For $\sph^2$, the metric tensor in spherical
coordinates $(\theta,\varphi)$, with $\theta$ being the colatitude and
$\varphi$ being the longitude, has the form
\[ g_{ij} = \begin{pmatrix} 1 & 0 \\ 0 & \sin^2 \theta \end{pmatrix}.
\] The metric tensor for $\sph^n$ also can be expressed in a similar
set of coordinates. On any Riemannian manifold there are two important
operators: the covariant derivative $\nabla$, powers of covariant
derivatives $\nabla^k$, which is an operator on tensors, and the
Laplace-Beltrami operator $\Delta = -\nabla^\ast \nabla$. The
covariant derivative operating on a function is the usual gradient,
expressed appropriately. Other powers are tensor operators. For
example, $\nabla^2$ plays the role of a Hessian. In local coordinates,
$\Delta$ has the form
\begin{align}
\Delta u = \frac{1}{\sqrt{\det(g_{ij})}} \sum_{i,j}
\frac{\partial}{\partial x^i} \sqrt{\det(g_{ij})} 
g^{ij} \frac{\partial u}{\partial x^i},
\label{lap_bel_op}
\end{align}
where $g^{ij} = (g_{ij})^{-1}$. For $\sph^2$, in spherical coordinates, the Laplace-Beltrami operator is given by
\[
\Delta u = \frac{1}{\sin \theta}\frac{\partial}{\partial \theta } \bigg(\sin \theta  \frac{\partial u}{\partial \theta } \bigg) + \frac{1}{\sin^2 \theta}\frac{\partial^2 u}{\partial \varphi^2 }.
\]

We now turn to a discussion of spherical harmonics; the details may be found in \cite{Mueller-66-1}. Spherical harmonics are eigenfunctions of $\Delta$. On $\sph^n$, the  eigenvalues of $-\Delta$ are $\lambda_\ell = \ell(\ell + n-1)$. The eigenspace corresponding to $\lambda_\ell$ is degenerate, and has dimension 
\begin{equation}
\label{dim_ell}
d_\ell = 
\left\{
\begin{array}{cc}
1,&\ell=0,\\[6pt]
\displaystyle{\frac{(2\ell+n-1) 
\Gamma(\ell+n-1)}{\Gamma(\ell+1)\Gamma(n)}} \sim \ell^{n-1}\,,
&\ell\ge 1\,.
\end{array}\right. .
\end{equation}
We note that $d_\ell = \calo(\ell^{n-1})$. The eigenfunctions corresponding to $\lambda_\ell$ are denoted by $Y_{\ell,k}$, where $k=1, \ldots, d_\ell$. We will use the real-valued versions of the spherical harmonics. The eigenspace of $\lambda_\ell$ will be denoted by $\calh_\ell$. The space of all spherical harmonics of order $L$ or less will be denoted by $\Pi_L = \bigoplus_{\ell=0}^L \calh_\ell$. In addition, we mention the well-known addition formula. Let $x,y\in \sph^n$ and let $x\cdot y$ denote the usual dot product from $\RR^{n+1}$. Then,
\begin{equation}
\label{addition_thm}
\sum_{k=1}^{d_\ell}Y_{\ell,k}(x)Y_{\ell,k}(y) = \frac{2\ell+n-1}{(n-1)\omega_n} P_\ell^{\frac{n-1}2}(x\cdot y),
\end{equation}
where  $\omega_n$ is the volume of $\sph^n$, and $ P_\ell^{\frac{n-1}2}$ is the degree $\ell$ ultraspherical polynomial of order $\frac{n-1}2$. 

The spherical harmonics form a complete orthonormal set in $L^2(\sph^n)$. Given $f$ in $L^2(\sph^n)$, we can expand $f$ in the series
$
f = \sum_{\ell=0}^\infty \sum_{k=1}^{d_\ell} \hat f_{\ell,k}Y_{\ell,k}.
$
As usual, given $f,g\in L^2$, we have
$
\langle f, g\rangle_{L^2} = \sum_{\ell=0}^\infty \sum_{k=1}^{d_\ell} \hat f_{\ell,k} \overline{\hat g_{\ell,k}}.
$

In this paper we will work with fractional order Sobolev spaces defined in terms of Bessel potentials. \cite{Strichartz-83-1, triebel1992}. The Sobolev space of order $\tau\ge 0$ is
\[
H_\tau := \{f\in L^2 \colon \|f\|_{H_\tau} := \|(I-\Delta)^{\tau/2} f\|_{L^2} <\infty\}.
\]
This is a Hilbert space in the inner product
\begin{equation}
\label{sobolev_inner_prod}
\langle f, g\rangle_{H_\tau} = \langle (I-\Delta)^{\tau/2}f, (I-\Delta)^{\tau/2}g\rangle_{L^2} = \sum_{\ell=0}^\infty \sum_{k=1}^{d_\ell} (1+\lambda_\ell)^\tau \hat f_{\ell,k} \overline{\hat g_{\ell,k}}.
\end{equation}
When $\tau=m$ is an integer, these spaces agree, up to norm equivalence,  with $W_2^m(\sph^n)$, which are defined in terms of covariant derivatives \cite{aubin1982, triebel1992}.

\section{Spherical Basis Functions and Approximation Spaces}\label{SBFs}

In the following, we will be working on $\sph^n$. We start with zonal functions. A continuous function $\phi:[-1,1]\to \RR$ is said to be \emph{zonal} if it has the expansion,
\begin{equation}
\label{zonal_function}
\phi(t):=\sum_{\ell=0}^\infty \hat\phi_\ell \frac{2\ell+n-1}{(n-1)\omega_n} P_\ell^{\frac{n-1}2}(t),
\end{equation}
where $P_\ell^{\frac{n-1}2}$ is a degree $\ell$ ultra spherical polynomial \cite[\S 4.7]{Szego-75-1}.  Zonal functions give rise to kernels on the sphere in the following way. Let $t=x\cdot y$, $x,y\in \sph^n$. Using the addition formula \eqref{addition_thm} in \eqref{zonal_function}, we see that
\begin{equation}
\label{zonal_kernel}
\phi(x\cdot y) = \sum_{\ell=0}^\infty  \hat\phi_\ell \sum_{k=1}^{d_\ell}Y_{\ell,k}(x)Y_{\ell,k}(y),
\end{equation}
which is a kernel mapping $\sph^n\times \sph^n$ to $\RR$.

SBFs are zonal functions having $\hat \phi_\ell>0$ for all $\ell\ge 0$. They are strictly positive definite functions on $\sph^n$. This means that the matrix $A = (\phi(\xi\cdot \eta))_{\xi,\eta \in X} $ is positive definite for every choice of $X$. Equivalently, 
\[
\sum_{\xi,\eta \in X}c_\xi c_\eta \phi(\xi\cdot \eta) > 0,
\]
except when the $c$'s are all $0$. $A$ being positive definite allows us to interpolate arbitrary continuous functions (or data, for that matter) on $\sph^n$, using functions from the \emph{approximation space} 
\[
V_{\phi,X}:= \spn\{\phi((\cdot)\cdot  \xi) \colon \xi \in X\}.
\]
This follows because the existence of $c=A^{-1}f|_X$ implies that 
\[
I_Xf(x) = \sum_{\xi \in X} c_\xi \phi(x\cdot \xi)
\]
interpolates $f$ on $X$, and does so uniquely. The SBF $\phi$ is also a reproducing kernel for the Hilbert space 
\[
\caln := \{f\in L^2 \colon \sum_{\ell,k} {\hat \phi_\ell}^{-1} |\hat f_{\ell,k}|^2<\infty\}.
\]
This space is often called the \emph{native space} of $\phi$; it has the inner product 
\begin{equation}
\label{native_space_inner_prod}
\langle f,g\rangle_\caln = \sum_{\ell=0}^\infty \sum_{k=1}^{d_\ell}\hat \phi_\ell^{-1}\hat f_{\ell,k} \overline{\hat g}_{\ell,k}.
\end{equation}

The SBF $\phi_\tau$ for which $\hat \phi_{\tau,\ell}=(1+\lambda_\ell)^{-\tau}$, with $\tau>n/2$, is especially important. The native space for $\phi_\tau$ is the Sobolev space $H_\tau$, since $\| f \|_\caln = \|(I-\Delta)^{\tau/2}f\|_{L^2}= \| f \|_{H_\tau}$. Making use of this observation yields a fractional order  ``zeros lemma,''  similar to integer order ones proved in \cite[Appendix A]{HNW_3_2011}. This will be important in the sequel.

\begin{lemma}[Zeros Lemma]  
\label{frac_zeros_lemma}
Let $\sigma,\tau\in \RR$ satisfy $\tau>n/2$, $0 \leq \sigma \leq \tau$. In addition, let $X\subset \sph^n$ be quasi uniform. If  $u\in\calh_\tau$ satisfies  $u|_X=0$, then, for $h_X$ sufficiently small, we have
\[
\|u\|_{H_\sigma} \le Ch_X^{\tau - \sigma}\|u\|_{H_\tau}.
\]
\end{lemma}

\begin{proof}
Let $I_{X, \,\phi_\tau}$ be the interpolation operator corresponding to $\phi_\tau$. By \cite[Theorem~5.5]{Narcowich-etal-07-1}, we have that
\[
\|u - I_{X,\,\phi_\tau}u\|_{H_\sigma} \le Ch_X^{\tau-\sigma}\|u\|_{H_\tau} .
\]
Note that $u|_X=0$ implies that $I_{X,\,\phi_\tau}u \equiv 0$. Using this in the previous equation then yields the result.
\end{proof}

The SBFs discussed above are all strictly positive definite. We will also need to make use of \emph{conditionally} positive definite SBFs. These SBFs have the form given in \eqref{zonal_function}, but the $\hat \phi_\ell$'s need only be positive for $\ell >L$. For $0\le \ell \le L$, $\hat \phi_\ell$ is arbitrary. Conditionally positive definite SBFs are employed to interpolate scattered data, with the requirement that the interpolants reproduce $\Pi_L$, the space of spherical harmonics of degree $L$ or less.  (Other spaces are also possible.)  For a conditionally positive definite SBF $\phi$, 
the corresponding approximation space  is defined to be
\[
V_{\phi,L,X} := \bigg\{\sum_{\xi\in X}a_\xi \phi((\cdot)\cdot \xi) \colon \sum_{\xi\in X}a_{\xi}\,p(\xi) =
0 \ \forall\ p\in \Pi_L \ \bigg\}+\Pi_L
\]
The interpolation operator  that both interpolates continuous functions and reproduces
$\Pi_L$ is
\begin{equation}
\label{SPD_SBF_interp_op}
I_{X,L}f = \sum_{\xi\in X}a_{\xi,L}\phi(x\cdot \xi)+p_{X,L} , \ p_{X,L} \in
\Pi_L, \ \text{ and } \sum_{\xi\in X}a_{\xi,L}p(\xi) =
0,\ p\in \Pi_L.
\end{equation}
The coefficients $a_{\xi,L}$ and the polynomial $p_{X,L}$ are determined by the requirements that the interpolation condition $I_{X,L}f|_X=f|_X$ hold and also that the coefficients satisfy the condition on the right above. Again, the interpolant is unique. There is also a semi-Hilbert space $\caln$ associated with $\phi$. This is defined to be
\[
\caln := \{f\in L^2 \colon \sum_{\ell=L+1}^\infty \sum_{k=1}^{d_\ell} {\hat \phi_\ell}^{-1} |\hat f_{\ell,k}|^2<\infty\},\ \langle f,g\rangle_\caln = \sum_{\ell=L+1}^\infty \sum_{k=1}^{d_\ell}\hat \phi_\ell^{-1}\hat f_{\ell,k} \overline{\hat g}_{\ell,k}.
\]
In addition, we will need the following well-known fact, which we state without proof.

\begin{proposition}
\label{phi_sobolev_space}
Suppose that $\tau>n/2$ and that $\phi$ is an SBF such that there are constants $c$, $C$ and $L\in \nats$ for which
$c(1+\lambda_\ell)^{-\tau} \le \hat \phi_\ell \le C(1+\lambda_\ell)^{-\tau}$ holds either for all $\ell\ge 0$ or for all $\ell \ge L+1$. If $\epsilon>0$, then $\phi \in H_{2\tau- \frac{n}{2} -\epsilon}$ . 

\end{proposition}


\section{Highly Localized ``Small Footprint'' Bases}
\label{small_footprint_bases}

\emph{Surface splines}, or \emph{polyharmonic kernels}, are special conditionally positive definite SBFs, and are a key ingredient in the kernel methods presented here. While they can be defined for any $\sph^n$ \cite{HNW_3_2011}, we will restrict our attention to the case of $\sph^2$ \cite{FHNWW2012}. Their explicit forms are given below: 

\begin{equation}\label{TPS}
\left.
\begin{array}{l}
\phi_m(t)= (-1)^{m}(1-t)^{m-1}\log(1-t), \ 1< m \in\nats, \\ [10pt]
\ \hat \phi_{m,\ell}=
  C_m\frac{\Gamma(\ell-m+1)}{\Gamma(\ell+m+1)}\sim \ell^{-2m} \sim \lambda_\ell^{-m}, \ \ell>m-1.
\end{array}
\right\}
\end{equation}
Here $C_m=2^{m+1}\pi \Gamma(m)^2$. These kernels are conditionally positive definite, and interpolation with them will reproduce $\Pi_{m-1}$. The space $\caln$ associated with $\phi_m$ is, up to norm equivalence, the Sobolev space $H_m(\sph^2)$ modulo  $\Pi_{m-1}$. Also, since $\hat \phi_{m,\ell} \sim \ell^{-2m}$, it is easy to show that $\phi_m \in H_{2m-1-\epsilon}(\sph^2)$, $\epsilon >0$. Furthermore, this implies that the approximation space for $\phi_m$ satisfies
\[
V_{\phi_m,X} := V_{\phi_m, m-1,X} \subset H_{2m-1-\epsilon}(\sph^2), \ \forall\ \epsilon >0.
\]

\subsection{Lagrange functions}
\label{lagrange_functions}

We can form a basis for $V_{\phi_m,X}$ using \emph{Lagrange functions} or \emph{cardinal functions}. A Lagrange function $\chi_\xi$ is defined as the unique interpolant from $V_{\phi_m,X}$ that satisfies $\chi_\xi(\eta)=\delta_{\xi,\eta}$. Since $\chi_\xi\in V_{\phi_m,X}$, it has the form
\begin{equation}
\label{chi_xi_expan}
\chi_\xi = \sum_{\zeta\in X} \alpha_{\xi,\zeta}\phi_m((\cdot)\cdot \zeta)+p_\xi, \ p_\xi \in \Pi_{m-1}, \ \sum_{\zeta\in X} \alpha_{\xi,\zeta}\,p(\zeta) =0\ \forall \ p\in \Pi_{m-1}.
\end{equation}
Interpolation using the $\chi_\xi$'s is simple: If $f$ is a continuous function on $\sph^2$, with $f|_X$ given, then $
I_{\phi_m,X}f = \sum_{\xi\in X}f(\xi)\chi_\xi$. 

There are two important properties of the Lagrange functions constructed from the $\phi_m$'s. First, they are well localized in space. Indeed, $\chi_\xi(x)$ decays exponentially in $\dist(x,\xi)$. Second, they have a small ``footprint'' in the set of basis elements. Again, the coefficients $\alpha_{\xi,\zeta}$ decay exponentially in  $\dist(\xi,\zeta)$. Each $\chi_\xi$ is effectively using only a small number of kernels from the set $\{\phi_m((\cdot)\cdot \xi) \colon \xi \in X\}$; i.e, $\chi_\xi$ has a small ``footprint'' in the set of kernels. The precise result is stated below:

\begin{theorem}[{\cite[Theorem~5.3]{FHNWW2012}}]
\label{main} 
Let $\rho>0$ be a fixed mesh ratio and let $\nats \ni m  \ge 2$. There exist constants $h^*$, $\nu$, $c_1$, $c_2$ and $C$, depending only on $m$ and $\rho$,
so that if $h_X\le h^*$, then $\chi_{\xi}$ given in \eqref{chi_xi_expan} has these properties:
\begin{align}
 |\chi_{\xi}(x)| &\le
  C \exp\left(-\nu\frac{\d(x,\xi)}{h_X}\right), \label{lagrange_decay}\\
|\alpha_{\zeta,\xi}|
& \le C q_X^{2-2m} \exp{\left(-\nu
      \frac{\d(\xi,\zeta)}{h_X}\right)}, \label{Coeff}
  \\[3pt]
c_1q_X^{2/p} \|\bfbeta\|_{\ell_p(X)} &\le \big\|\sum_{\xi\in\Xi} \beta_{\xi}
  \chi_{\xi}\big\|_{L^p(\sph^2)} \le c_2 q_X^{2/p} \|\bfbeta\|_{\ell_p(X)}.
  \quad \label{p_stability}
\end{align}
\end{theorem}

In addition to the various bounds above, we will also need a bound on $\nabla \chi_\xi$, the covariant derivative of $\chi_\xi$. The lemma below will be needed to obtain this bound, as well as several others in the sequel.

\begin{lemma}
Let $x\in \sph^2$ be fixed. Then, there is a constant $C$ that is independent of $\nu$ and the properties of $X$ for which we have
\begin{equation}\label{basic_sum}
\sum_{\xi\in X} e^{-\frac{\nu}{h_X} \d(x,\xi)}  <  \frac{C\rho_X^2}{(1 - e^{-\nu})^2}.
\end{equation}
In addition, if $B(x,r_0)$ is the ball of radius $r_0$ and center $x$, then
\begin{equation}\label{truncated_sum}
\sum_{\xi\in X\cap B(x,r_0)^\complement} e^{-\frac{\nu}{h_X} \d(x,\xi)} <  C \rho_X^2\frac{ n_0e^{-(n_0-1)\nu} }{(1 - e^{-\nu})^2}, \ n_0=\lceil r_0/h_X\rceil. 
\end{equation}
\end{lemma}

\begin{proof}
Divide the sphere into bands of width $\sim h_X$, center $x$, and outer radius $\sim n h_X$, $n\ge 1$. The sum then satisfies the inequality
\[
\begin{aligned}
\sum_{\xi\in X} e^{-\frac{\nu}{h_X} \d(x,\xi)} &=  \sum_{n=1}^{n_{max}}  \sum_{x\in \text{band}_n\cap X} e^{-\frac{\nu}{h_X} d(x,\xi)}\\
&< \sum_{n=1}^{n_{max}} \# (\text{band}_n\cap X) e^{-(n-1)\nu },
\end{aligned}
\]
where $n_{max}\sim \pi/h_X$. The area of $\text{band}_n$ is $\sim nh_X^2$. Consequently, we have that cardinality $\#(\text{band}_n\cap X)$ is $\sim nh_X^2/q_X^2=n\rho_X^2$. Using this in the equation above yields
\[
\sum_{\xi\in X} e^{-\frac{\nu}{h_X} d(x,\xi)} < C\rho_X^2 \sum_{n=1}^\infty n e^{-(n-1)\nu }.
\]
Summing the series on the right above yields \eqref{basic_sum}. To obtain \eqref{truncated_sum}, we sum the series $\sum_{n=n_0}^\infty n e^{-(n-1)\nu }$ and use the fact that $n_0-(n_0-1)e^{-\nu}<n_0$.
\end{proof}

\begin{theorem} 
\label{grad_estimates}
Adopt the notation of the Theorem~\ref{main}. There exists a constant $C=C(\rho,m)$ such that
\begin{equation}
\label{grad_bnd_chi_xi}
|\nabla \chi_\xi(x)| \leq Cq_X^{-1} e^{-\frac{\nu}{h_X} d(x,\xi)}.
\end{equation}
In addition, $\|\nabla \chi_\xi\|_{L^\infty} \le Cq_X^{-1}$. Finally, 
\begin{equation}
\label{covariant_lebesgue_const}
\Lambda_1 := \max_{x\in \sph^2}\sum_{\xi \in X} |\nabla \chi_\xi(x)| < C\rho_X^2 q_X^{-1} \frac{1}{(1 - e^{-\nu})^2}.
\end{equation}
\end{theorem}

\begin{proof}
The H\"{o}lder estimate given in  \cite[Theorem~5.3]{HNW_3_2011}, with $\epsilon = 1$, is
\[
|\chi_\xi(x)-\chi_\xi(y)| \leq C \frac{d(x,y)}{q_X} e^{-\frac{\nu}{h_X} d(x,\xi)}.
\]
Fixing $x$ and dividing by $d(x,y)$ yields
$$\bigg| \frac{\chi_\xi(x)-\chi_\xi(y)}{d(x,y)}\bigg| \leq C q_X^{-1} e^{-\frac{\nu}{h_X} d(x,\xi)}.$$
Let $\hat{t}$ be a unit tangent vector based at $x$. Choose $y$ to be a point along the geodesic starting at $x$ with tangent $\hat{t}$. Then, 
\[
\lim_{d(x,y) \to 0} \bigg| \frac{\chi_\xi(x)-\chi_\xi(y)}{d(x,y)}\bigg| = |D_{\hat{t}}(\chi_\xi)(x)| \leq  Cq_X^{-1} e^{-\frac{\nu}{h_X} d(x,\xi)} .
\]
This holds for every direction $\mathbf t$. Since $\max_{\mathbf t} | D_{\mathbf t}(\chi_\xi)(x)| =|\nabla \chi_\xi(x)|$,  the bound  \eqref{grad_bnd_chi_xi} follows immediately. Obviously, we also have $\|\nabla \chi_\xi\|_{L^\infty} \le Cq_X^{-1}$.
\end{proof}

\begin{proposition} 
\label{sobolev_norm_chi_xi_chi_eta}
Adopt the notation and assumptions of Theorem~\ref{main} and suppose that $a,b \in C^\infty$, $m\ge 2$ and $0<\epsilon < 2m-3$. Then, $b\chi_\xi \chi_\eta \in H_{2m-1-\epsilon}\cap L^\infty$. Moreover, for $h_X$ sufficiently small, there exists $C=C(\rho, m)$ such that 
\begin{equation}
\label{sobolev_bnd_chi_xi_chi_eta}
\| b\chi_\xi \chi_\eta \|_{H_{2m-1-\epsilon}} \le Ch_X^{2+\epsilon - 2m} \|b\|_{H_{2m}},
\end{equation}
\begin{equation}
\label{sobolev_bnd_grad_dot_grad}
\| a\nabla \chi_\xi \cdot \nabla \chi_\eta \|_{H_{2m-\epsilon-2}} \le Ch_X^{1+\epsilon - 2m}\|a\|_{H_{2m}}.
\end{equation}
\end{proposition}

\begin{proof} In the proof below we will need the inequalities $2m-\epsilon-1>2$ and $2m-\epsilon-2>1$, which follows easily from $0<\epsilon < 2m-3$. 

Theorem~\ref{leibnitz_rule} applies to $b\chi_\xi$, because $b \in C^\infty \subset H_{2m-1-\epsilon}\cap L^\infty$ and, by Theorem~\ref{main},  $\chi_\xi \in H_{2m-1- \epsilon}\cap L^\infty$.  Consequently, the three products $b\chi_\xi\ \chi_\eta$, $b\chi_\xi$ and $\chi_\xi\chi_\eta$ are in $H_{2m-1-\epsilon}\cap L^\infty$. A straightforward application of Theorem~\ref{leibnitz_rule}, equation \eqref{leibinitz_rule_bnd}, to the various products then results in this bound:
\[
\| b \chi_\xi \chi_\eta \|_{H_{2m-1-\epsilon}} \le C'\|b\|_{L^\infty} \big(\| \chi_\xi\|_{L^\infty} \|\chi_\eta \|_{H_{2m-1-\epsilon}} +  
\| \chi_\xi \|_{H_{2m-1 -\epsilon}} \|\chi_\eta\|_{L^\infty}\big) + C\| b\|_{H_{2m-1- \epsilon}} \|\chi_\xi\|_{L^\infty}\| \chi_\eta\|_{L^\infty}.
\]
By \eqref{p_stability}, with $p=\infty$, we have that both $\|\chi_\xi\|_{L^\infty}$ and $\|\chi_\eta\|_{L^\infty}$ are bounded by the constant $c_2$, because the corresponding $\bfbeta$'s have a single entry, $1$. Moreover, since $2m-1 - \epsilon >1$, the Sobolev embedding theorem and a standard inclusion inequality imply that $\|b\|_{L^\infty} \le C\| b\|_{H_{2m-1 - \epsilon}}\le C\| b\|_{H_{2m}}$. Inserting these in the previous inequality then yields
\begin{equation}
\label{triple_bound}
\| b \chi_\xi \chi_\eta \|_{H_{2m-1 - \epsilon}} \le Cc_2\|b\|_{H_{2m}}\big(\|\chi_\xi\|_{H_{2m-1-\epsilon}}+\|\chi_\eta \|_{H_{2m - 1 -\epsilon}}+ c_2\big)
\end{equation}
We will now employ a Bernstein inequality\footnote{The precise version of the theorem holds for a positive definite SBF. However, it is easy to modify it so that it will hold for a conditionally positive definite SBF.} \cite[Theorem~6.1]{mhaskar-etal-2010} that holds for functions in $V_{\phi_m,X}$. The parameters in the theorem are $\beta=2m$, from \eqref{TPS}, $p=2$, $n=2$, $\gamma=2m-1-\epsilon >2$ and $g=\chi_\xi$. The theorem then implies that 
$\| \chi_\xi \|_{H_{2m-1-\epsilon}} \le Cq_X^{1+\epsilon - 2m} \| \chi_\xi \|_{L^2}$. Moreover, if we set $p=2$ in \eqref{p_stability}, we also have both $\|\chi_\xi\|_{L^2}$ and $\|\chi_\eta\|_{L^2}$ bounded by $c_2 q_X$. Thus, $\| \chi_\xi \|_{H_{2m-1- \epsilon}} \le Cq_X^{2+\epsilon - 2m} $. Combining the various bounds above we arrive at $\| b \chi_\xi \chi_\eta \|_{H_{2m-1 - \epsilon}} \le Cc_2\|b\|_{H_{2m}}q_X^{2+\epsilon-2m}(2+c_2q_X^{2m-2-\epsilon})$. Since $q_X\ll 1$ and $2m-2-\epsilon>1$, we have $ \| b \chi_\xi \chi_\eta \|_{H_{2m-1 - \epsilon}} \le Cc_2\|b\|_{H_{2m}}q_X^{2+\epsilon-2m}$. From this, \eqref{sobolev_bnd_chi_xi_chi_eta} follows on observing that $q_X\sim h_X$.

To obtain the second bound, note that, by Corollary~\ref{grad_squared_bnd}, the conditions on $\chi_\xi, \chi_\eta$ imply that $\nabla \chi_\xi\cdot \nabla \chi_\eta \in H_{2m - \epsilon -2}\cap L^\infty$, since $2m - \epsilon -2>1$, and that those on $a$ are the ones used for $b$. Consequently, 
\[
\| a \nabla \chi_\xi\cdot \nabla \chi_\eta\|_{H_{2m - \epsilon -2}} \le C\|a\|_{H_{2m}} \big( \|\nabla \chi_\xi \cdot \nabla \chi_\eta\|_{H_{2m - \epsilon -2}} + \| \nabla \chi_\xi \cdot \nabla \chi_\eta\|_{L^\infty} \big).
\]
Since $2m-2-\epsilon >1$, we may again apply the Sobolev embedding theorem to obtain $\| \nabla \chi_\xi \cdot \nabla \chi_\eta\|_{L^\infty} \le C \|\nabla \chi_\xi \cdot \nabla \chi_\eta\|_{H_{2m - \epsilon -2}} $. Combining this with the previous inequality results in 
\begin{equation}
\label{a_dot_prod_bnd}
\| a \nabla \chi_\xi\cdot \nabla \chi_\eta\|_{H_{2m - \epsilon -2}} \le C\|a\|_{H_{2m}} \|\nabla \chi_\xi \cdot \nabla \chi_\eta\|_{H_{2m - \epsilon -2}} .
\end{equation}
To estimate the norm on the right we will use Corollary~\ref{grad_squared_bnd}. This implies that 
\[
\|\nabla \chi_\xi \cdot \nabla \chi_\eta\|_{H_{2m - \epsilon -2}} \le C\big(\|\chi_\xi\|_{H_{2m-1-\epsilon}}+\|\chi_\eta \|_{H_{2m-1-\epsilon}}\big)\big(\|\nabla \chi_\xi\|_{L^\infty}+\|\nabla \chi_\eta \|_{L^\infty}\big).
\]
We may use Proposition~\ref{grad_estimates} and the bounds on $\|\chi_\xi\|_{H_{2m-1-\epsilon}}$, $\|\chi_\eta\|_{H_{2m-1-\epsilon}}$ found above to obtain this:
\begin{equation}
\label{dot_prod_bnd}
\|\nabla \chi_\xi \cdot \nabla \chi_\eta\|_{H_{2m - \epsilon -2}} \le Ch_X^{2+\epsilon - 2m}q_X^{-1}\le C h_X^{1+\epsilon -2m} 
\end{equation}
Finally, using the bound from \eqref{dot_prod_bnd} in \eqref{a_dot_prod_bnd} yields \eqref{sobolev_bnd_grad_dot_grad}.
\end{proof}

\subsection{Quadrature formulas}
\label{quadrature_formulas}

Numerically computing the integrals that arise in any Galerkin method ultimately requires a quadrature formula. In the setting of a sphere and other homogeneous manifolds, kernel quadrature formulas \cite{FHNWW2013, hesse-et-al-2010,sommariva_womer2005} have been developed and analyzed. Let $f$ be continuous and consider the surface spline $\phi_m$ given in \eqref{TPS}. In addition, let $Y$ be a quasi-uniform set of points on $\sph^2$, which may be different from $X$. The quantities $q_Y$, $h_Y$, and $\rho_Y$, and cardinality $N_Y$ have their usual meanings. Using $\phi_m$, form the Lagrange functions $\tilde \chi_\zeta$, $\zeta\in Y$ corresponding to $Y$ and the interpolant $I_Y\!f\!=\sum_{\zeta \in Y} f(\zeta)\tilde \chi_\zeta$. The quadrature formula  is obtained integrating $I_Y\!f$:
\[
Q_Y(f) = \int_{\sph^2}I_Y\!f(x)d\mu(x) =  \sum_{\zeta \in Y} f(\zeta) w_\zeta, \quad w_\zeta :=\int_{\sph^2}\tilde \chi_\zeta(x)d\mu(x).
\]

We point out that a few of the weights $w_\zeta$ can be near zero or
become slightly negative in the case of arbitrary $Y$.  This is
usually \emph{not }the case for most quasi-uniform sets $Y$. (See the
discussion in \cite[Section 2.2.1]{FHNWW2013,
  sommariva_womer2005}). In fact, not only are the weights positive
for most sets, but they also satisfy the lower bound
\begin{equation}
\label{lower_bnd_wgt}
w_\zeta \ge Ch_Y^2.
\end{equation}
In the rest of our discussion, we will assume that
\eqref{lower_bnd_wgt} holds. The only situation where this assumption
comes into play will be in stability considerations of the discretized
version of the stiffness matrix.

Positive or not, the weights all satisfy an \emph{upper} bound;
namely,
\begin{equation}
\label{upper_bnd_wgt}
|w_\zeta| \le Ch_Y^2.
\end{equation}
Since $w_\zeta=\int_{\sph^2}\tilde \chi_\zeta(x)d\mu(x)$, we have that
$|w_\zeta|\le \| \tilde \chi_\zeta \|_{L_1(\sph^2)}$. To estimate the
right side, use \eqref{p_stability}, with $p=1$, $\beta_\zeta=1$ and
all of the other $\beta$'s equal to $0$. This gives us $\|
\bfbeta\|_{\ell_1} = 1$, and so $\| \tilde \chi_\zeta
\|_{L_1(\sph^2)}\le c_2q_Y^2\le c_2\rho_Y^{-2}h_Y^2$.

The salient feature of this quadrature formula is that the weights can
be obtained by solving a linear system of equations that is stable
and, while not sparse, has entries that decay rapidly as they move
away from the diagonal. For $m=2$, weights for a set $Y$ having
600,000 points were easily computed \cite[section 5]{FHNWW2013}.

Error estimates for the quadrature formula $Q_Y$ have been derived for
functions in various integer valued Sobolev spaces. However, we will
need stronger results. We begin with the proposition below, which
holds on $\sph^n$, $n\ge 2$. Consider a (conditionally) positive
definite SBF $\phi$ that satisfies $c(1+\lambda_\ell)^{-\tau} \le \hat
\phi_\ell \le C(1+\lambda_\ell)^{-\tau}$ for all $\ell\ge L+1$. The
novel feature of this result is that it uses a new version of the
``doubling trick,'' which is established in
Theorem~\ref{cond_pos_def_case}, to obtain higher convergence rates
for functions smoother than ones in the native space of $\phi$. The
result is this:

\begin{proposition} Let $\tau>n/2$, $2\tau\ge \mu>n/2$, and  $f\in H_{\mu}$. If $h_Y$ is sufficiently small,  then 
\begin{equation}
\label{quad_sobolev_est}
\bigg| \int_{\sph^n}f(x)d\mu - Q_Y(f)\bigg| \le C h^{\mu}_Y \|f\|_{H_{\mu}}.
\end{equation}
\end{proposition}

\begin{proof}
Note that $\big|\int_{\sph^n} fd\mu - Q_Y(f)\big| \le \int_{\sph^n} |f-I_Y\!f |d\mu \le \omega_n^{1/2}\|f-I_Y\!f \|_{L^2}$, where $\omega_n$ is the volume of $\sph^n$. In addition, by Theorem~\ref{cond_pos_def_case}, with $\beta =0$, we have  $\|f-I_Y\!f \|_{L^2}\le C h^{\mu}_Y \|f\|_{H_{\mu}}$. Combining the two inequalities yields \eqref{quad_sobolev_est}.
\end{proof}

\section{Weak and Strong Solutions to $Lu=f$}
\label{Lu=f_prop}

In the section we will lay out the properties, assumptions and various aspects of weak and strong solutions to \eqref{pde_dyadic}. In local coordinates on $\sph^2$, this equation has the form,
\begin{align}
Lu = -\frac{1}{\sqrt{\det(g_{ij})}} \sum_{i,j}\frac{\partial}{\partial x^i} \sqrt{\det(g_{ij})} 
a^{ij}(x) \frac{\partial u}{\partial x^j} + b(x)u =f.
\label{Ltensor}
\end{align}
Here, $g_{ij}$ is the covariant form of the standard metric tensor $g$ on $\sph^2$; as usual, $g^{ij}=[g_{ij}]^{-1}$ are the contravariant components of $g$. The $a^{ij}$'s are contravariant components of a $C^\infty$, symmetric rank 2 tensor $a$ that is positive definite in the sense that there exist positive constants $c_1,c_2$ such that
\begin{align}
c_1 \sum_{i,j}g^{ij}(x)v_i v_j \leq \sum_{i,j} a^{ij}(x)v_i v_j \leq c_2 \sum_{i,j}g^{ij}(x)v_iv_j
\label{a_tensor_pos}
\end{align}
holds for all vectors $v$ in the tangent space at $x \in \sph^2$. The function $b(x)$ is $C^\infty$. In addition, we assume that there are constants $b_1,b_2$ such that, for all $x\in \sph^2$, $0<b_1 \le b(x) \le b_2$.   We note that in the case that $a=g$, this reduces to the case $Lu = -\Delta u + b u$.

With $L$ as given in \eqref{Ltensor} and  $f\in L^2(\sph^2)$, we can place $Lu=f$ into weak form by multiplying by $v \in H_1$  and integrating by parts to arrive at
\begin{align}
\langle u,v\rangle_a := \int_{\sph^2} \bigg(\sum_{i,j=1}^2 a^{ij} \frac{\partial u}{\partial x^i} \frac{\partial v}{\partial x^j}  + b u v \bigg)d\mu = \int_{\sph^2} fv d\mu := \ell(v).
\label{weaktensor}
\end{align}
By \eqref{a_tensor_pos} and the assumptions on $b(x)$, the bilinear form $\langle \cdot,\cdot \rangle_a $ satisfies 
\begin{align}
M_1\|u\|_{H_1}^2 =(c_1+b_1) \langle u,u\rangle_{H_1} \leq \underbrace{\langle u,u\rangle_a }_{\|u\|_a^2}\leq (c_2+b_2)  \langle u,u \rangle_{H_1} =M_2\|u\|_{H_1}^2
\label{a_tensor_norm}
\end{align}
A straightforward application of the Lax-Milgram theorem, together with $\ell(v)$ being a bounded linear functional on $H_1$, then yields the following result:

\begin{proposition}
\label{a_coercive}
The bilinear form $\langle \cdot,\cdot \rangle_a $ is coercive and bounded on $H_1$ and defines an inner product on $H_1$, with the norms $\|\cdot \|_a $ and $\| \cdot \|_{H_1}$ being equivalent. In addition, for $f\in L^2$, there is a unique $u\in H_1$ such that \eqref{weaktensor} is satisfied; that is, $u \in H_1$ weakly solves $Lu=f$. Finally, $\|u\|_{L^2}\le \|f\|_{L^2}$.
\end{proposition}

We now turn to the regularity of the weak solution to $Lu=f$.  \emph{A priori} estimates of the general type needed here may be found in the survey article by Mikhailets and Murach \cite[Theorem~6.6]{Mikhailets-Murach-2012}, along with references. They are, however, given for pseudo-differential operators. A simpler approach is to use the local regularity theorems in \cite[pgs.\ 261-269]{Folland_book_1976}, which  apply to open sets in $\RR^n$, and so to coordinate patches on $\sph^n$. Since the sphere is compact, they apply globally to $Lu=f$, and so we have the (standard) regularity result that we will use in the sequel.

\begin{proposition} 
\label{regularity}
Let $L$ be as described above. If $u$ is a distributional solution to $Lu=f$, where $f\in H_s$, $0\le s$, $s\in \RR$, then for any $t<s-1$ there is a constant $C_t>0$ such that $u\in H_{s+2}$ and $\|u\|_{H_{s+2}} \le C_t (\|Lu\|_{H_{s}}+  \|u\|_{H_t})$. In addition, we have that $\|u\|_{H_{s+2}} \le C\|Lu\|_{H_s}$.
\end{proposition}

\begin{proof}
The assertions in \cite[Corollary 6.27 and Theorem 6.30]{Folland_book_1976} regarding regularity and the inequality $\|u\|_{H_{s+2}} \le C_t (\|Lu\|_{H_{s}}+  \|u\|_{H_t})$ are true for elliptic operators in general, and specifically for our $L$, which is strongly elliptic and has $C^\infty$ coefficients. To obtain the second inequality, start by setting $t=0$ in the first inequality. Also note that, from \eqref{weaktensor}, we have $\min(b) \|u\|_{L^2}^2 \le \|u\|_a^2  = \langle Lu,u \rangle_{L^2} \le \|Lu\|_{L^2}\|u\|_{L^2}$. Dividing by $\|u\|_{L^2}$, we obtain $\|u\|_{L^2} \le C\|Lu\|_{L^2}$, since $\min(b)>0$. Thus $\|u\|_{s+2} \le C\big( \|Lu\|_{H_s} + \|Lu\|_{L^2}\big)$. The inequality we want then follows from the observation that $\|Lu\|_{L^2}\le \|Lu\|_{H_s}$.
\end{proof}

We close this section with a corollary to the regularity result above. The corollary forms the basis of the ``Nitsche trick" \cite{Nitsche-71} that we will use later.

\begin{corollary}
\label{nitsche_prelim} Let $V$ be a closed subspace of $H_1$ and let $P_V$ be the orthogonal projection of $V$ onto $H_1$, relative to the inner product $\langle \cdot,\cdot \rangle_a$. If $u\in H_1$ and $Lw=u-P_V u$, then 
\begin{equation}
\label{nitsche_ineq}
\|u-P_V u\|_{L^2}^2 \le \|w-P_V w\|_{a}\|u-P_V u\|_{a}
\end{equation}

\begin{proof}
Since we have $Lw=u-P_V u$, the regularity result above implies that $w\in H_3$. Integrating by parts in $\langle w, u-P_V u \rangle_a$ yields $\langle w, u-P_V u \rangle_a = \langle Lw, u-P_V u\rangle_{L^2} = \langle u-P_V u, u-P_V u\rangle_{L^2}= \|u-P_V u\|_{L^2}^2$. Next, note that $P_V w$ is in V, and so $P_V w$ is orthogonal to $u-P_V u$, relative to $\langle \cdot,\cdot \rangle_a$. Consequently, $\langle w, u-P_V u \rangle_a = \langle w -P_V w, u-P_V u \rangle_a$.  It follows that $\|u-P_V u\|_{L^2}^2=  \langle w -P_V w, u-P_V u \rangle_a$. Applying Schwarz's inequality then yields \eqref{nitsche_ineq}.
\end{proof}
\end{corollary}

\section{Galerkin Approximation for $Lu=f$}
\label{continuous_galerkin_approx}

\subsection{Error estimates}
\label{SBF_tau_est}

We will use spaces of spherical basis functions to obtain approximate solutions to $Lu=f$; specifically, the $V_{\phi,X}$'s  and the $V_{\phi,L,X}$'s defined earlier. Let $\phi$ be an SBF on $\sph^2$ that is  positive definite or conditionally positive definite. For $\tau>1$, we will make the assumption that the Fourier-Legendre coefficients of $\phi$ satisfy 
\begin{equation}
\label{tau_assumpt}
c(1+\lambda_\ell)^{-\tau} \le \hat \phi_\ell \le C(1+\lambda_\ell)^{-\tau}, \ \forall \ \ell\ge L+1,
\end{equation} 
where $c$ and $C$ are positive constants and $L$ is the highest order special harmonic reproduced by interpolation from $V_{\phi,L,X}$.  In later sections, when we will be concerned with the stability of discretizing the problem, we will restrict the SBFs to the thin-plate splines. For obtaining error estimates, this is unnecessary.  

Let $P_X := P_{V_X}$ be the orthogonal projection of $H_1$ onto the finite dimensional space $V_X$, in the $\langle \cdot,\cdot \rangle_a$ inner product, and let $I_X$ be the interpolation operator associated with $V_X$. Since $\|u-P_Xu\|_a= \min_{v\in V_X}\|u-v\|_a$, we have that 
\begin{equation}
\label{a_norm_bound}
\|u-P_Xu\|_a \le \|u-I_X u\|_a \le C\|u-I_X u\|_{H_1},
\end{equation}
where the last inequality follows from the equivalence of the norms $\|\cdot \|_a$ and $\| \cdot \|_{H_1}$. The same reasoning applies to the solution $w$ to $Lw=u-P_X u$, so $\|w-P_Xw\|_a \le C\|w-I_Xw \|_{H_1}$. Combining these estimates with the one from Corollary~\ref{nitsche_prelim} then yields this:
\begin{equation}
\label{L2_H1_estimate}
\|u-P_Xu\|_{L^2}^2 \le C \|w-I_X w\|_{H_1}\|u-I_X u\|_{H_1}.
\end{equation}

The regularity results in Proposition~\ref{regularity} imply that if $f\in H_s$ then the solution $u$ to $Lu=f$ is in $H_{s+2}$. Moreover, if $\tau >1$, then the projection $P_X u $ exists and is in $H_{\tau +\alpha}$, for any $\alpha < \tau-1$. Thus, $u-P_Xu$ belongs to $H_\sigma$, $\sigma := \min(s+2,\tau+\alpha)$. Applying the elliptic regularity result to $Lw=u-P_Xu$ then gives us $w\in H_{\sigma+2}$. 

We are mainly interested in the case of $\sph^2$ -- i.e., $n=2$. For that case, we have the following lemma, which will be needed to obtain error estimates.

\begin{lemma}
Let $n=2$. In the notation used above, 
\begin{equation}
\label{w_interp_u_interp_bound}
\|w-I_X w\|_{H_1} \le C h^2 \|u - I_Xu\|_{H_1}.
\end{equation}
\end{lemma}

\begin{proof}
The solution to $Lw = u-P_Xu$ is in $w\in H_{\sigma+2}$, where $\sigma := \min(s+2,\tau+\alpha)$, $\alpha < \tau-n/2= \tau-1$, and $\tau>n/2=1$.  It follows that $\sigma >1$ and so $w$ is in $H_{\sigma+2} \subset H_3$. Applying  Theorem~\ref{cond_pos_def_case} then yields
\begin{equation}
\label{H_1_H_3_bound}
\|w-I_X w\|_{H_1} \le Ch^2 \|w\|_{H_3}.
\end{equation}
Furthermore, by Proposition~\ref{regularity}, $\|w\|_{H_3} \le C\|Lw\|_{H_1}=\|u-P_X u\|_{H_1}$. Since the usual norm for $H_1$ is equivalent to the $\|\cdot \|_a$ norm, we have $\|w\|_{H_3} \le C\|u-P_X u\|_a$. Then, by this inequality and \eqref{a_norm_bound}, we see that $\|w\|_{H_3} \le C\|u- I_Xu\|_{H_1}$. Combining this with \eqref{H_1_H_3_bound} gives us \eqref{w_interp_u_interp_bound}.
\end{proof}

\begin{theorem}
\label{u_L2_proj_bound_thm}
Let $n=2$ and $Lu=f$, $f\in H_s$, $s\ge 0$. In the notation used above, 
\begin{equation}
\label{u_L2_proj_bound}
\|u-P_X u\|_{L^2} \le \left\{\begin{array}{cl} C h^{s+2} \|u\|_{H_{s+2}}\le  C h^{s+2} \|f\|_{H_s}, & \text{if }\ s\le 2\tau -2, \\
C h^{2\tau} \|u\|_{H_{2\tau}} \le C h^{2\tau} \|f\|_{H_{2\tau-2}}, & \text{if }\ 2\tau -2 <s.
\end{array}
\right.
\end{equation}
\end{theorem}

\begin{proof} By Theorem~\ref{cond_pos_def_case}, we have that
\begin{equation}
\label{u_H1_interp_bound}
\|u-I_X u\|_{H_1} \le \left\{\begin{array}{cl} C h^{s+1} \|u\|_{H_{s+2}} & \text{if }\ s\le 2\tau -2, \\
C h^{2\tau-1} \|u\|_{H_{2\tau}} & \text{if }\ 2\tau -2 <s.
\end{array}
\right.
\end{equation}
The estimate in terms of the Sobolev norms of $u$  then follows from \eqref{L2_H1_estimate} and \eqref{w_interp_u_interp_bound}.  Proposition~\ref{regularity} implies that $\|u\|_{s+2}\le C\|Lu\|_s$. Since $f=Lu$, $\|u\|_{s+2}\le C\| f \|_s$. When $s>2\tau - 2$, then the previous argument applies, with $s+2$ being replaced by $2\tau$.
\end{proof}

\subsection{The stiffness matrix in the Lagrange basis}
\label{A_lagrange_basis}

The error estimates obtained above are, as we noted, independent of the SBF used. However, to actually solve for the Galerkin approximation, we must pick a suitable $\phi$ for which there is a good basis for $V_{\phi,L,X}$, one that results in a numerically robust method for finding the the Galerkin solution. We will show that a surface spline $\phi_m$ and the corresponding Lagrange basis $\{\chi_\xi\colon \xi \in X\}$ will provide the required robustness.

Let $m\ge 2$ and set $V_{m,X}=V_{\phi_m,m-1,X}$. Take the basis for $V_{m,X}$ to be $\{\chi_\xi\colon \xi \in X\}$. The Galerkin approximation to the solution $Lu=f$ is the orthogonal projection $u_X:=P_{m,X}u$ of $u$ onto $V_{m,X}$, in the $\langle \cdot,\cdot\rangle_a$ inner product. If $u_X=\sum_{\xi\in X} \alpha_\xi \chi_\xi$, then, from  the weak form of $Lu=f$ and the usual normal equations, we obtain the \emph{stiffness matrix}:
\begin{equation}
\label{stiff_def}
A\alpha = \tilde f, \ \text{where} \ A_{\xi,\eta} = \langle \chi_\xi, \chi_\eta\rangle_a, \ \alpha = (\alpha_\xi), \ \tilde f=(\langle f, \chi_\xi\rangle_{L^2}).
\end{equation}
Eventually, we will discretize the problem by using quadrature methods to approximate $A$. For now, we will restrict our attention to $A$. 

\subsubsection{Stability of $A$}

We want to estimate $\kappa_2(A)$, the condition number for $A$. We begin with the observation that $A$ is a real, self-adjoint matrix. It is also a Gram matrix for the linearly independent set, $\{\chi_\xi\colon \xi \in X\}$, and is therefore positive definite as well. Consequently, $\kappa_2(A) = \frac{\lambda_{\max}(A)}{\lambda_{\min}(A)}$. 

We will begin by estimating $\lambda_{\min}(A)$. First of all, the operator $L$ is self adjoint and positive definite. Standard variational methods then imply that 
\begin{equation}
\label{min_eig_L}
\min_{v\in H_1}\langle v,v\rangle_a =\lambda_{\min}(L), \  \|v\|_{L^2}=1.
\end{equation}
Let $v=\sum_{\xi\in X} \alpha_\xi \chi_\xi$. Consider the quadratic form $\alpha^T A \alpha = \langle v,v\rangle_a$. By \eqref{min_eig_L}, we have that 
\[
\alpha^T A \alpha \ge \lambda_{\min}(L) \|v\|^2_{L^2} .
\]
Next, from \eqref{p_stability}, with $p=2$, we have $\|v\|_{L^2} = \|\sum_{\xi\in X} \alpha_\xi \chi_\xi\|_{L^2} \ge c_1 q_X \|\alpha\|_{\ell_2(X)}$. From the inequality above, we then have 
\[
\alpha^T A \alpha \ge \lambda_{\min}(L) \|v\|^2_{L^2} \ge c_1^2 q_X^2 \lambda_{\min}(L) \|\alpha\|^2_{\ell_2(X)}, 
\]
which holds for all $\alpha \in \RR^{|X|}$. Hence, we have that
\begin{equation}
\label{min_eig_A}
\lambda_{\min}(A) \ge c_1^2 q_X^2 \lambda_{\min}(L) .
\end{equation}

Estimating $\lambda_{\max}(A)$ requires the Bernstein inequality from \cite[Theorem~6.1]{mhaskar-etal-2010}. In the case at hand, the theorem cited\footnote{The precise version of the theorem holds for a positive definite SBF. However, it is easy to modify it so that it will hold for a conditionally positive definite SBF.} holds with $n=2$, $p=2$, $0<\gamma<2m-1$, and $g=v$. For every $v\in V_{m,X}$ and every $0<\gamma<2m -1$, we have that the is a constant $C>0$, where $C=C(m,\rho)$, such that
\begin{equation}
\label{bernstein}
\|g\|_{H_\gamma} \le Cq_X^{-\gamma} \|g\|_{L^2}.
\end{equation}
Since $\alpha^T A \alpha = \langle v,v\rangle_a$ and, by \eqref{a_tensor_norm}, $\langle v,v\rangle_a\le M_2 \|v\|_{H_1}^2$, then, from \eqref{bernstein} ($\gamma=1$) and  \eqref{p_stability} ($p=2,n=2$), we have that
$
\alpha^T A \alpha \le Cq_X^{-2} \|v\|_{L^2}^2 \le Cc_2\| \alpha \|^2_{\ell_2(X)}.
$
Consequently, 
\begin{equation}
\label{max_eig_A}
\lambda_{\max}(A) \le C, \ C=C(m,\rho,\| \cdot \|_a ).
\end{equation}
Combining \eqref{max_eig_A} and \eqref{min_eig_A} results in the following:

\begin{theorem}\label{stiffness_stability}
Let $A$ be the stiffness matrix in the basis $\{\chi_\xi\colon \xi\in X\}$ for $V_{m,X}$. If $q_X$ is sufficiently small and $\rho$ is fixed, then there is a constant $C=C(m,\rho,\| \cdot \|_a)$ for which the condition number $\kappa_2(A)$ satisfies $\kappa_2(A) \leq Cq_X^{-2}$.
\end{theorem}

\subsubsection{Exponential decay of the entries of $A$}

At this point, we turn to the behavior of the entries $A_{\xi,\eta}$ in $A$. What we will see is that the entries in $A$ decay exponentially in $\dist(\xi,\eta)$, making $A$ nearly sparse. In section~\ref{sparse_app_local_lag}, we will use this decay to construct a sparse discretization for $A$. Establishing decay requires the following lemma.

\begin{lemma}
\label{schwarz_lemma_a_g}
Let $f,g$ be in $C^1(\sph^2)$, and let $\nabla f, \nabla g$ be their covariant derivatives. Then, there is a constant $C$ such that for all $x\in \sph^2$ we have
\begin{equation}
\label{a_inner_prod_g_norm}
\big|\sum_{i,j}a^{ij}(x) \nabla_i f (x)\nabla_j g(x)\big| \le C|\nabla f (x)| \,|\nabla g(x)|.
\end{equation}
\end{lemma}

\begin{proof}
The matrix $a^{ij}$ is positive definite, so may use it as an inner product. Schwarz's inequality applied to this inner product implies that 
\[
\big|\sum_{i,j}a^{ij}\nabla_i f \nabla_j h \big| \le\big (\sum_{i,j}a^{ij}\nabla_i f\nabla_j f\big)^{1/2}(\sum_{i,j}a^{ij}\nabla_i  g\nabla_j g\big)^{1/2}
\]
By \eqref{a_tensor_pos}, we have that $\sum_{i,j}a^{ij}\nabla_i f\nabla_j f\le \sum_{i,j}g^{ij}\nabla_i f\nabla_j f=|\nabla f|^2$. This also holds for $g$ as well. Applying these inequalities  to the previous one yields \eqref{a_inner_prod_g_norm}.
\end{proof}

\begin{proposition} 
\label{siff_decay_A}
For $h_X$ sufficiently small, 
\begin{equation}
\label{stiff_exp_decay}
|A_{\xi,\eta}| \le Ch_X^{-2}e^{-\frac{\nu}{h_X}\,\dist(\xi,\eta)}.
\end{equation}

\end{proposition}

\begin{proof}
Since we have $A_{\xi,\eta} = \langle \chi_\xi,\chi_\eta \rangle_a$, we have that 
\[
|A_{\xi,\eta}| \le \int_{\sph^2} \big|\sum_{i,j}a^{ij}\nabla_i \chi_\xi \nabla_j \chi_\eta + b(x)  \chi_\xi \chi_\eta\big| d\mu(x).
\]
By  Lemma~\ref{schwarz_lemma_a_g} and the boundedness of $b$, we have that 
\[
\big|\sum_{i,j}a^{ij}\nabla_i \chi_\xi \nabla_j \chi_\eta + b(x)  \chi_\xi \chi_\eta\big| \le C|\nabla \chi_\xi | |\nabla \chi_\eta | + \|b\|_{L^\infty} |\chi_\xi |\,|\chi_\eta |.
\]
Moreover, using \eqref{grad_bnd_chi_xi} and \eqref{lagrange_decay}, we see that 
\[
C|\nabla \chi_\xi | \, |\nabla \chi_\eta | + \|b\|_{L^\infty} |\chi_\xi |\,|\chi_\eta | \le (C_1q_X^{-2}+C_2)\exp\left(-\nu\frac{\d(x,\xi) +\d(x,\eta )}{h_X}\right).
\]
Because $\dist(\cdot,\cdot)$ is a metric on $\sph^2$, we may use the triangle inequality: $\d(x,\xi) +\d(x,\eta )\ge \d(\xi,\eta)$. In addition, $h_X$ small implies that $q_X =h_X/\rho_X$ is also small. Combining these facts yields, uniformly in $x$, this inequality:
\[
\big| \sum_{i,j}a^{ij}\nabla_i (\chi_\xi) \nabla_j (\chi_\eta) + b(x)  \chi_\xi \chi_\eta\big| \le Ch_X^{-2} e^{-\frac{\nu}{h_X} \d(\xi,\eta)} .
\]
Integrating both sides above then establishes \eqref{stiff_exp_decay}. 
\end{proof}

\section{The Discretized Galerkin Solution}
\label{discretized_galerkin_solution}

The discretized solution to the Galerkin problem is obtained simply by replacing the stiffness matrix $A$ from the original problem with a discretized version, which is obtained via quadrature, and then solving as usual. In order to carry out a complete analysis of this method, in the sequel we will restrict the tensor $a^{ij}$ to have the form $a^{ij} = a\,g^{ij}$, where $a\in C^\infty(\sph^2)$.

\subsection{Discretizing the stiffness matrix}

We now turn to the task of discretizing the stiffness matrix. Our approach is to approximate the $(\xi,\eta)$ entry $A_{\xi,\eta}= \int_{\sph^2}\big(a\nabla \chi_\xi\cdot \nabla \chi_\eta+ b\chi_\xi\chi_\eta\big)d\mu$ by means of the  quadrature formulas discussed in section~\ref{quadrature_formulas}.  In doing so, we will allow for the surface spline used in the Galerkin method, $\phi_m$, to differ from the one used in the quadrature formula. We will denote the latter by $\phi_M$, with $M\ge 2$.

\subsubsection{Discretization error for the stiffness matrix}

To discretize the stiffness matrix, we will employ a set of nodes $Y$ that is chosen independently of $X$. In general, $Y$ will be much larger than $X$ and need not contain $X$ as a subset. That said, the discretization of $ A_{\xi,\eta}=\int_{\sph^2}\big(a\nabla \chi_\xi\cdot \nabla \chi_\eta+ b\chi_\xi\chi_\eta\big)d\mu$  is  $Q_Y(a\nabla\chi_\xi\cdot \nabla \chi_\eta+b\chi_\xi\chi_\eta)$. In explicit form, this is given by 
\begin{equation}
\label{discretized_entry}
A^Y_{\xi,\eta}:= \sum_{\zeta \in Y} \big(a\nabla\chi_\xi\cdot \nabla \chi_\eta+b\chi_\xi\chi_\eta \big)\big|_\zeta \,w_\zeta.
\end{equation}
We will need the following lemma to obtain bounds on the error $|A_{\xi,\eta}- A^Y_{\xi,\eta}|$:
\begin{lemma}
\label{optimize_mu}
Suppose that $m > M$ are positive integers and that $\epsilon>0$. Let $\sigma=m-M$. If $M\ge \sigma+1$ and $\epsilon<2\sigma-1$, then $\epsilon<m-2$ and $2m - 1 - \epsilon >2M$.
\end{lemma}

\begin{proof} 
We will begin by showing show the first inequality. By assumption, $\epsilon <2\sigma-1$, so $\epsilon <\sigma +\sigma+1-2\le \sigma+M-2=m-2$. To get the second, note that $2m-\epsilon -1 =2M+ 2\sigma -1 -\epsilon$. Since $2\sigma-1-\epsilon>0$, we have that $2m-\epsilon -1 >2M$.
\end{proof}

\begin{corollary}
\label{quad_diff_stiff}
Suppose that  $a(x),b(x) \in C^\infty(\sph^2)$, and adopt the notation and assumptions from Lemma~\ref{optimize_mu}. In addition, let $\{\chi_\xi \colon \xi \in X\}$ and $\{\tilde \chi_\xi \colon \xi \in Y\}$ be the Lagrange bases for $V_{m,X} := V_{\phi_m,X}$ and for $V_{M,Y} := V_{\phi_M,Y}$, respectively. Then, if $\delta:=2\sigma -1 - \epsilon>0$, we have:
\begin{equation}
\label{quad_diff_stiff_bnd}
\big| A_{\xi,\eta}- A^Y_{\xi,\eta}\big| \le 
C(h_Y/h_X)^{2M} h_X^{-\delta},\ C=C(\|a\|_{H_{2m}},\|b\|_{H_{2m}}),
\end{equation}
where $A^Y_{\xi,\eta} = Q_Y(a\nabla\chi_\xi\cdot \nabla \chi_\eta+b\chi_\xi\chi_\eta)$. Moreover, $\delta\le m-2-\epsilon$, so it may be made as small as we wish by taking $\epsilon$ close to $m-2$.  
\end{corollary}

\begin{proof}
By Proposition~\ref{sobolev_norm_chi_xi_chi_eta}, $b\chi_\xi\chi_\eta$ is in $H_{2m-1-\epsilon}$ for all $0<\epsilon <m-2$. By Lemma~\ref{optimize_mu}, if $M\ge \sigma+1$ and $\epsilon<2\sigma -1$, then we have both $2m-1-\epsilon >2M$ and $\epsilon<m-2$. It follows that we may use \eqref{quad_sobolev_est} with $\tau=M$ and $\mu = 2M<2m-1-\epsilon=2M+2\sigma-1-\epsilon$. From the bound in \eqref{sobolev_bnd_chi_xi_chi_eta}, and from  $ \|b\chi_\xi\chi_\eta\|_{H_{2M}} \le \|b\chi_\xi\chi_\eta\|_{H_{2m-1-\epsilon}}$, we see that 
\begin{align}
\bigg| \int_{\sph^2} b\chi_\xi\chi_\eta d\mu - Q_Y(b\chi_\xi\chi_\eta)\bigg| &\le 
C  h_Y^{2M} \|b\chi_\xi\chi_\eta\|_{H_{2m-1-\epsilon}}  \nonumber \\ 
&\le C h_Y^{2M} h_X^{1+\epsilon -(2\sigma-1)-2M}\|b\|_{H_{2m}} \nonumber \\
&\le C(h_Y/h_X)^{2M} h_X^{1-\delta}\|b\|_{H_{2m}}. 
\label{LS_scalar_est}
\end{align}
Using the same argument, but with the bounds from \eqref{sobolev_bnd_grad_dot_grad} instead of \eqref{sobolev_bnd_chi_xi_chi_eta}, we have 
\begin{equation}
\bigg| \int_{\sph^2} a\nabla\chi_\xi\cdot \nabla \chi_\eta d\mu - Q_Y(a\nabla\chi_\xi\cdot \nabla \chi_\eta)\bigg| \le 
C(h_Y/h_X)^{2M} h_X^{-\delta}\|a\|_{H_{2m}}.  \label{LS_vector_est}
\end{equation}
If we combine  \eqref{LS_scalar_est} and \eqref{LS_vector_est} and note that $h_X<\pi$, we obtain  \eqref{quad_diff_stiff_bnd}. To prove the statement concerning $\delta$, observe that $\delta= \sigma -2 -\epsilon + (\sigma +1)\le m-M-2+\epsilon +M=m-2 - \epsilon$. 
\end{proof}

\begin{remark}\em If $M\ge m$ and $0<\epsilon<m-2$ , then, obviously, $1<m+1<2m-\epsilon-1 <2M$. Thus, we may use \eqref{quad_sobolev_est} with $\tau=M$ and $\mu = 2m-1-\epsilon$. The same arguments employed above then imply that the error estimate in \eqref{quad_diff_stiff_bnd} becomes
\begin{equation}
\label{quad_diff_stiff_bnd_M=m}
\big| A_{\xi,\eta}- A^Y_{\xi,\eta}\big|\le C\max(\|a\|_{H_{2m}},\|b\|_{H_{2m}}) (h_Y/h_X)^{2m-1-\epsilon}.
\end{equation}
The right side in this inequality depends on $m$ and $\epsilon$, but not on $M$. It follows that there is no advantage in choosing $M>m$.\em
\end{remark}

One important fact is that $A^Y_{\xi,\eta}$ decays in $\d(\xi,\eta)$ in the same way as  $A_{\xi,\eta}$. We will establish this below. Before carrying out the proof, we mention that, although we use the assumption \eqref{lower_bnd_wgt} in our proof, it is not necessary to do so. We also wish to point out that the denominator of the fraction in the exponent is $h_X$, and not $h_Y$, as one would first suppose.

\begin{proposition}
\label{decay_AY}
The discretized entry $A^Y_{\xi,\eta}$ satisfies the bound $\big| A^Y_{\xi,\eta}\big|  \le Ch_X^{-2}\exp(-\frac{\nu}{h_X}\d(\xi,\eta))$.
\end{proposition}

\begin{proof}
The same argument used to establish \eqref{stiff_exp_decay} yields $|A^Y_{\xi,\eta}| \le Ch_X^{-2}\big(\sum_{\zeta\in Y}|w_\zeta|  \big) e^{-\frac{\nu}{h_X}\d(\xi,\eta)}$. By our assumption that the weights are positive, we have that $\sum_{\zeta\in Y}|w_\zeta |  = \sum_{\zeta\in Y} w_\zeta =4\pi$, from which the result is immediate.
\end{proof}

The theorem below gives us the desired bound on the error $\|A-A^Y \|_2$ that is made in using quadrature to compute the entries in the stiffness matrix.  
\begin{theorem}
\label{error_l2_norm_A}
Let $m> M\ge 2$ and let $\phi_M$ be the surface spline used for quadrature. If  $0<\delta < 2(m-M) -1$, then, for $h_X$ and $h_Y$ sufficiently small, we have
\begin{equation}
\label{error_l2_bound}
\|A-A^Y\|_2 \le C(\log(h_Y))^2 (h_Y/h_X)^{2M} h_X^{-\delta}
\end{equation}
\end{theorem}

\begin{proof}
Recall that, for a self-adjoint matrix $C$, $\|C\|_2 \le \|C\|_1 = \|C\|_\infty$. Applying this to the self-adjoint matrix $A-A^Y$ yields
\[
\|A-A^Y\|_2 \le \max_{\eta\in X} \big(\textstyle{\sum_{\xi\in X}} |A_{\xi,\eta} - A^Y_{\xi,\eta}|\big).
\]
The task is now to bound the sums on the right above. Let $r_0>h_x$ and let $B(\eta,r_0)$ be the ball with center $\eta$ and radius $r_0$. We may break up the sum over $X$ into a sum over centers inside $B(\eta,r_0)$ and those in $B(\eta,r_0)$: 
\[
\sum_{\xi\in X} |A_{\xi,\eta} - A^Y_{\xi,\eta}| = \sum_{\xi\in X\cap B(\eta,r_0)} |A_{\xi,\eta} - A^Y_{\xi,\eta}| + \sum_{\xi\in X\cap B(\eta,r_0)^\complement} |A_{\xi,\eta} - A^Y_{\xi,\eta}|. 
\]
By Propositions \ref{siff_decay_A} and \ref{decay_AY}, we have that 
$
\big| A_{\xi,\eta} - A^Y_{\xi,\eta} \big| \le Ch_X^{-2}e^{-\frac{\nu}{h_X}\d(\xi,\eta)}.
$
Using this and \eqref{truncated_sum} yields 
\begin{equation}
\label{contribution_dist_xi}
\sum_{\xi\in X\cap B(\eta,r_0)^\complement} |A_{\xi,\eta} - A^Y_{\xi,\eta}| \le Ch_X^{-2}\sum_{\xi\in X\cap B(\eta,r_0)^\complement}e^{-\frac{\nu}{h_X}\d(\xi,\eta)} < C \rho_X^2h_X^{-2}\, \frac{ e^{-(n_0-1)\nu}n_0 }{(1 - e^{-\nu})^2},\ n_0=\lceil r_0/h_X\rceil .
\end{equation}
The set of remaining centers  is $X\cap B(\eta,r_0)$, whose cardinality may be bounded by $\mathrm{vol}( B(\eta,r_0))/\mathrm{vol}( B(\eta,q_X))\sim r_0^2/q_x^2=\rho_X^2 (r_0/h_X)^2<\rho_X^2 n_0^2 $. From this fact and the uniform estimate on $|A_{\xi,\eta}-A^Y_{\xi,\eta}|$ in \eqref{quad_diff_stiff_bnd}, we see that
\begin{equation}
\label{contribution_near_xi}
\sum_{\xi \in X\cap B(\eta,r_0)}\big| A_{\xi,\eta} - A^Y_{\xi,\eta} \big| \le C\rho_X^2 n_0^2 (h_Y/h_X)^{2M} h_X^{-\delta}.
\end{equation}
Choose a constant $K$ so that $K\nu>2M$ and pick $r_0=Kh_X|\log(h_Y)|$ so that $n_0\sim K |\log(h_Y)|$. (The $h_Y$ is \emph{not} a mistake.) The bounds  in \eqref{contribution_dist_xi} and \eqref{quad_diff_stiff_bnd} are then $Ch_X^{-2}h_Y^{K\nu} |\log (h_Y)|$ and $C(\log(h_Y))^2  (h_Y/h_X)^{2M} h_X^{-\delta}$, respectively. Note that $h_X^{-2} h_Y^{K\nu} \le h_Y^{2M} h_X^{-2M} \le (h_Y/h_X)^{2M}$. Since $h_Y$ and $h_X$ are small and $\delta>0$, we also have that $|\log(h_Y)|\le (\log(h_Y))^2h_X^{-\delta}$. Combining the various bounds above results in $\sum_{\xi \in X}\big| A_{\xi,\eta} - A^Y_{\xi,\eta} \big| \le C(\log(h_Y))^2 (h_Y/h_X)^{2M} h_X^{-\delta}$. This holding uniformly in $\eta$ immediately implies \eqref{error_l2_bound}.
\end{proof}

\subsubsection{Stability of the discretized stiffness matrix}

We now turn to the question of how stable, numerically,  the discretized stiffness matrix $A^Y$ is.  Answering this question requires the following lemma, which relates certain $L^2$ norms. We will need the two lemmas below.

\begin{lemma}
\label{L2_X_Y_bnds}
Let $X$, $Y$ be quasi-uniform, with $\rho_X,\rho_Y\le \rho$ and suppose that $u\in V_X$ and $I_Yu$ is the interpolant of $u$ relative to the kernel $\phi_m$, $m\ge 2$ and its associated space $V_Y$. Then, there exists a constant $C(\rho)$ such that $\frac12 \|u\|_{L^2}\le \| I_Yu\|_{L^2}\le \frac32 \|u\|_{L^2}$, provided $h_Y \le C(\rho)q_X$. 
\end{lemma}

\begin{proof}
Let $\tilde u = I_Yu$. Since $\|u\|_{L^2} - \|u- \tilde u\|_{L^2}\le \|\tilde u\|_{L^2}\le \|u\|_{L^2} +\|u- \tilde u\|_{L^2}$, we need only find $C(\rho)$ such that $\|u- \tilde u\|_{L^2}\le \frac12 \|u\|_{L^2}$. From \cite[Theorem~4.6]{FHNWW2013}, we have that
\[
\|u- \tilde u\|_{L^2}=\|u- I_Yu\|_{L^2} \le C_1h_Y^2 \|u\|_{W_2^2},
\]
since $u\in W_2^3\subset  W_2^2$ and $m\ge \mu =  2$.  The constant $C_1$ only depends on $Y$ through $\rho$. We now apply the Bernstein inequality\footnote{The theorem actually requires the SBFs involved to be strictly positive definite. However, by a simple adaptation of the argument used to prove \cite[Theorem~4.6]{FHNWW2013}, one can establish the result needed here.}  from \cite[Theorem~6.1]{mhaskar-etal-2010}. First, $u\in V_X$, the space associated with $\phi_3$. This is essentially the Green's function associated with $\beta=6$ in the Bernstein inequality. In addition, we may take $\gamma=2$. Thus, we have
\[
\|u\|_{W_2^2}\approx \|u\|_{H_2}  \le C_2q_X^{-2}\|u\|_{L^2}, \text{where }  C_2=C_2(\rho).
\]
Combining this with the previous inequality yields $\|u- \tilde u\|_{L^2} \le C_1C_2(h_Y/q_X)^2\|u\|_{L^2}$. The result immediately follows on choosing $C(\rho) \le (2C_1C_2)^{-1/2}$.
\end{proof}

\begin{theorem}
\label{norm_inv_A_Y}
Let $a,b\in C^\infty$ satisfy $a(x)\ge a_0>0$ and $b(x)\ge b_0>0$. Then $\lambda_{\rm min}(A^Y) \ge Cq_X^2$, provided $h_Y \le C(\rho)q_X$.
\end{theorem}

\begin{proof}
Let $u\in V_X$, and so,  $u=\sum_{\xi \in X} u(\xi)\chi_\xi$ and $\nabla u = \sum_{\xi\in X} u(\xi )\nabla \chi_\xi$. Moreover, $u|_X$ is an arbitrary vector in $\RR^{N_X}$. It follows that 
\[
\begin{aligned}
(u|_X)^TA^Yu|_X &= \sum_{\zeta \in Y} (a(\zeta) \nabla u(\zeta)\cdot \nabla u(\zeta) + b(\zeta) u(\zeta)^2)w_\zeta \\
&\ge b_0 \sum_{\zeta \in Y}u(\zeta)^2 w_\zeta \ge Cb_0N_Y^{-1} \sum_{\zeta \in Y}u(\zeta)^2 \ \text{(by \eqref{lower_bnd_wgt})} \\
&\ge Cb_0h_Y^2 \|u|_Y\|_{\ell^2}^2.
\end{aligned}
\]
Let $\tilde u = I_Yu = \sum_{\zeta\in Y} u(\zeta)\tilde \chi_\zeta$, which is the interpolant of $u$ relative to $V_Y$, the space associated with $\phi_m$ and $Y$. Of course, since $\tilde u$ is the interpolant of $u$ on $Y$, $\tilde u|_Y  = u|_Y$. By Corollary~3.11, (3.1) and (3.3) in \cite{HNSW_2_2011},  $\|u|_Y\|_{\ell^2} = \| \tilde u|_Y\|_{\ell^2} \ge Cq_Y^{-1} \| \tilde u\|_{L^2}$. In addition, Lemma~\ref{L2_X_Y_bnds} implies that $\| \tilde u\|_{L^2}= \| I_Yu\|_{L^2} \ge \frac12 \| u\|_{L^2}$. Again applying the results from \cite{HNSW_2_2011} then yields $\|u\|_{L^2} \ge Cq_X \| u|_X\|_{\ell^2} $.  Consequently, $\|u|_Y\|_{\ell^2} \ge Cq_Y^{-1}q_X \| u|_X\|_{\ell^2} $. Combining this with the lower bound on the quadratic form $(u|_X)^TA^Yu|_X $ then gives us $(u|_X)^TA^Yu|_X \ge Cb_0(h_Y/q_Y)^2 q_X^2 \| u|_X\|_{\ell^2}^2 = Cb_0 \rho^2 q_X^2 \| u|_X\|_{\ell^2} $, so $\lambda_{\rm min}(A^Y)\ge Cq_X^2$.
\end{proof}

\begin{corollary}\label{conditioning_A_Y}
If $h_Y$ is chosen so that $\log(h_Y))^2 (h_Y/h_X)^{2M} h_X^{-\delta}\le C$, then the condition number $\kappa_2(A^Y)$ satisfies $\kappa_2(A^Y) \le Cq_X^{-2}$.
\end{corollary}

\begin{proof}
Because $\lambda_{\rm max}(A^Y)= \|A^Y\|_2 \le \|A\|_2+\|A-A^Y\|_2=\lambda_{\rm max}(A)+\|A-A^Y\|_2$, we have, from \eqref{max_eig_A}, \eqref{error_l2_bound} and the condition on $h_Y$, that $\lambda_{\rm max}(A^Y) \le C$. Applying Theorem~\ref{norm_inv_A_Y} then yields the result.
\end{proof}

\subsection{Error estimates for the discretized Galerkin solution}

Let $f|_X$  to be the restriction of $f$ to the set $X$ and $u_h:=u_{h_X}=\sum_\xi \alpha_\xi \chi_\xi$ be the  Galerkinn approximation to the solution $u$ of $Lu=f$. The coefficient vector $\alpha$ is given by $\alpha=A^{-1}f|_X$. The discretized solution $u^Y_h:=u^Y_{h_X}$ is obtained by replacing the stiffness matrix $A$ by $A^Y$ in the problem. The solution that results is $u^Y_h= \sum_\xi \alpha^Y_\xi \chi_\xi$, where  $\alpha^Y=(A^Y)^{-1}f|_X$. 

Our goal is to analyze the $L^2$ error between $u$ and $u^Y_h$, The triangle inequality implies that $\|u-u^Y_h\|_{L^2}\le \|u-u_h\|_{L^2}+\|u_h-u^Y_h\|_{L^2}$. We can estimate $\|u-u_h\|_{L^2}$ using \eqref{u_H1_interp_bound}:
\begin{equation}
\label{continuous_galerkin_bnd}
\|u-u_h\|_{L^2} \le  Ch^{s+2}\|f\|_{H_s},
\end{equation}
We also have, by  \eqref{p_stability} and $A^{-1}-(A^Y)^{-1} = (A^Y)^{-1} (A^Y-A)A^{-1}$, that
\begin{align}
\|u_h-u^Y_h\|_{L^2} &= \|\textstyle{\sum_\xi}(\alpha_\xi-\alpha^Y_\xi) \chi_\xi  \|_{L^2} \nonumber\\
&\le c_2q_X \|\alpha-\alpha^Y\|_{\ell^2} \nonumber \\
& \le c_2 q_X \|(A^Y)^{-1}\| \|A^Y-A\|\, \|\underbrace{A^{-1}f|_X}_{\alpha}\|_{\ell^2}. \nonumber
\end{align}
Using  \eqref{p_stability} again, we have $\|\alpha\|_{\ell^2} \le c_1^{-1}q_X^{-1}\|\sum_\xi a_\xi \chi_\xi\|_{L^2}=c_1^{-1}q_X^{-1}\|u_h\|_{L^2}$. In addition, from Theorem~\ref{norm_inv_A_Y}, $ \|(A^Y)^{-1}\| \le Cq_X^{-2}$. Combining these inequalities results in
\[
\|u_h-u^Y_h\|_{L^2}  \le Cq_X^{-2} \|A^Y-A\| \| u_h\|_{L^2}
\]
Because  $\|u-u_h\|_{L^2} \le  Ch^{s+2}\|f\|_{H_s}$, we have $\|u_h\|_{L^2} \le \|u\|_{L^2}+Ch_X^{s+2}\|f\|_{H_s}$. Moreover,  by Proposition~\ref{a_coercive}, $\|u\|_{L^2} \le C\|f\|_{L^2}\le C\|f\|_{H_s}$. Thus, for $h_X\sim q_X$ small,
\[
\|u_h-u^Y_h\|_{L^2}  \le Cq_X^{-2} \|A^Y-A\| \big(\|f\|_{H_s} +Ch_X^{s+2}\|f\|_{H_s} \big)\le Cq_X^{-2} \|A^Y-A\| \|f\|_{H_s}.
\]
From this and \eqref{continuous_galerkin_bnd}, it follows that
\begin{equation}
\label{discrete_galerkin_sum_error}
\|u-u^Y_h\|_{L^2}\le C\big( h_X^{s+2}+ q_X^{-2} \|A^Y - A\|_2 \big) \|f\|_{H_s}
\end{equation}
The above discussion together with Theorem \ref{error_l2_norm_A} yields \eqref{discrete_galerkin_bnd} below. Note that the second term in \eqref{discrete_galerkin_bnd} measures the quadrature error. Ideally, the ``fine set'' $Y$ can be chosen so that the second term is comparable to the optimal error $O(h_X^{s+2})$.
\begin{theorem}
\label{discrete_galerkin_error}
Let $f,s$ be as in Theorem~\ref{u_L2_proj_bound_thm} and let $m,M,\delta$ be as in Theorem~\ref{error_l2_norm_A}. Then,  for $h_X$ and $h_Y$ sufficiently small, we have
\begin{equation}
\label{discrete_galerkin_bnd}
\|u-u^Y_h\|_{L^2}\le C\bigg(h_X^{s+2}+\frac{(\log(h_Y) h_Y^{M})^2}{h_X^{2M+2+\delta}}\bigg) \|f\|_{H_s}.
\end{equation}
\end{theorem}

\section{Sparse Approximation and Local Lagrange Functions}
\label{sparse_app_local_lag}

This section discusses reducing the computational expense of numerically finding the Galerkin solution to the problem. There are two aspects of this. The first is obtaining a truncated approximation $\widetilde A^Y$ to the discretized stiffness matrix $A^Y$. Each row in $\widetilde A^Y$ has $\calo(\log(N_X))^2)$ nonzero entries, as opposed to $N_X$ in $A^Y$. The second is to replace the global Lagrange basis with a local one. As mentioned in the introduction, computing this basis requires inverting a number small matrices, a task that is  parallizable. The error estimates from making these approximations are virtually unchanged.

\subsection{Sparse Approximation}

So far, we have not addressed the question of how well a \emph{sparse}
approximation to the stiffness matrix would perform. Suppose that, in
$A^Y$, we discard all entries $A^Y_{\xi,\eta}$ that satisfy
$\dist(\xi,\eta)\ge Kh_X|\log h_X|$, where $K\nu>2$. Let the matrix we
get in this way be $\widetilde A^Y$, where
\begin{equation}
\label{truncated_A} \widetilde A^Y_{\xi,\eta}:=\begin{cases} 0, &
\dist(\xi,\eta) > Kh_X|\log h_X| \, , \\ A^Y_{\xi,\eta}, &
\dist(\xi,\eta) \le Kh_X|\log h_X| .
\end{cases}
\end{equation}

The matrix $\widetilde A^Y_{\xi,\eta}$ is symmetric. The number of
nonzero elements in each row is approximately the ratio of the areas
of caps having radii $Kh_X|\log h_X|$ and $h_X$, respectively. If we
make use of this and of the fact that, since $X$ is quasi uniform,
$h_X\sim N_X^{-1/2}$, then we see that
\begin{equation}
\label{row_cardinality}
\#\{\text{row}\  \eta\} \approx \frac{(Kh_X|\log h_X|)^2}{h_X^2} = K^2(\log(h_X))^2 \approx \frac14 K^2(\log(N_X))^2,
\end{equation}
as opposed to $N_X$ for $A^Y$ itself.

\begin{proposition}\label{sparse_approx_prop} 
Let $K\nu>2$ and $\widetilde A^Y$ be defined by \eqref{truncated_A}. Then,
\begin{equation}
\label{sparse_approx_est}
\| A^Y - \widetilde A^Y\|_2 \le \frac{2CKe^{-\nu}}{(1-e^{-\nu})^2}h_X^{K\nu - 2}|\log(h_X)|.
\end{equation}
\end{proposition}
  
\begin{proof}
We will follow the proof of Theorem~\ref{error_l2_norm_A}. Because $ A^Y$ and $\widetilde A^Y$ are symmetric, the norm $\| A^Y - \widetilde A^Y\|_2$ satisfies the bound
\[
\| A^Y - \widetilde A^Y\|_2\le \| A^Y - \widetilde A^Y\|_1= \| A^Y - \widetilde A^Y\|_\infty = \max_{\eta\in X} \big(\textstyle{\sum_{\xi\in X}} |A^Y_{\xi,\eta} - \widetilde A^Y_{\xi,\eta}|\big).
\]
We again want to estimate each term in the sums above. Let $B_\eta$ be the ball centered at $\eta$ and having radius $r_h = Kh_X|\log(h_X)|$. From Proposition~\ref{decay_AY} the definition of $ \widetilde A^Y$, we have that
\[
\sum_{\xi\in X}|A^Y_{\xi,\eta} - \widetilde A^Y_{\xi,\eta}| = \sum_{\xi\in X\cap B_\eta^\complement} |A^Y_{\xi,\eta}| \le Ch_X^{-2}
\sum_{\xi\in X\cap B_\eta^\complement} e^{-\frac{\nu}{h_X}\d(\xi,\eta)}.
\]
Next, divide $B_\eta^\complement$ into bands of width $\sim h_x$, the $n^{th}$ band being a distance approximately $r_h+nh_X$ from $\eta$. Repeating the derivation of \eqref{contribution_dist_xi}, \emph{mutatis mutandis}, we obtain
\[
\sum_{\xi\in X}|A^Y_{\xi,\eta} - \widetilde A^Y_{\xi,\eta}|\le Ch_X^{-2}e^{-\nu K|\log(h_X)|}\sum_{n=1}^\infty (K|\log(h_X)|+n)e^{-\nu n}\le \frac{2CKe^{-\nu}}{(1-e^{-\nu})^2}h_X^{K\nu-2}|\log(h_X)|.
\]
Combining the inequalities above yields \eqref{sparse_approx_est}.
\end{proof}

\subsection{Local Lagrange Functions}\label{Local Lagrange Functions}
The result above quantifies the error made in zeroing out the entries $A^Y_{\xi,\eta}$ corresponding to all $\xi,\eta$ such that $\dist(\xi,\eta) > Kh_X|\log h_X|$. To obtain the rest of the entries we still need to use \eqref{discretized_entry}; this entails finding the $\chi_\xi$'s, which are global in the sense that they require all of the points in $X$ for their computation. 

There is a way around this. In \cite{FHNWW2012}, Fuselier \emph{et al}.\ introduced a basis for $V_{\phi_m,X}$ composed of \emph{local} Lagrange functions, $\{\chi^{loc}_\xi\colon \xi \in X\}$. These basis functions are simply Lagrange functions for the points in $X$ that lie in a ball of radius $Kh_X|\log(h_)|$ about $\xi\in X$. A detailed description of their construction and properties is given in \cite[\S 6.3]{FHNWW2012}. We will simply list what we need here, in the theorem below.

\begin{theorem}[{\cite[Theorem~6.5]{FHNWW2012}}]\label{loc_lag_properties}
 Let the notation and assumptions of Theorem~\ref{main} hold; define $\Upsilon_\xi := X \cap B(\xi,Kh_X|\log (h_X) |)$.  There exists \footnote{One may take $\mu=\iota$, where $\iota$ is constructed  in \cite[Lemma~6.4]{FHNWW2012}.}  $\mu = \mu(m)$ such that for $K>0$ satisfying $J:= K\nu -4m + 2-2\mu>0$ these  these hold:
\begin{gather}
\|  \chi^{loc}_\xi - \chi_\xi\|_{L_\infty} \le C\ h_X^{J}, \label{error_loc_chi} \\
| \chi^{loc}_{\xi}(x)| \le C\big(1+\d(x,\xi)/h_X\big)^{-J}. \nonumber 
\end{gather}
Furthermore, when $J>2$, the set $\{\chi^{loc}_\xi\}$ is $L^p$
stable: there are $C_1,C_2>0$ for which
\begin{equation*}
\label{p_stable_loc_chi_J}
C_1 q_X^{2/p} \|\bfbeta\|_{\ell^p(X)}
\le 
\big\|\textstyle{\sum_{\xi\in X}} \beta_{X} \chi^{loc}_{\xi}\big\|_{L^p(\sph^2)}
\le 
C_2 q_X^{2/p} \|\bfbeta\|_{\ell^p(X)}.
\end{equation*}
\end{theorem}

Local Lagrange functions may be expanded in global ones. Because $\chi^{loc}_\xi$ is a Lagrange function for
$\Upsilon_\xi$, it satisfies $\chi^{loc}_\xi(\eta) = \delta_{\xi,\eta}$, for
$\eta\in \Upsilon_\xi$.  Of course, we also have $\chi_\xi(\eta) =
\delta_{\xi,\eta}$, for all $\eta\in X$. Expanding $\chi^{loc}_\xi(x)$
in terms of the basis $\{\chi_\eta\}_{\eta\in X}$ results in
\begin{equation}
\label{local2global}
\chi^{loc}_\xi(x) = \chi_\xi(x) + \sum_{\eta\not\in
\Upsilon_\xi}\chi^{loc}_\xi(\eta) \chi_\eta(x) = \chi_\xi(x) +
\sum_{\eta\not\in
\Upsilon_\xi}\big(\chi^{loc}_\xi(\eta)-\chi_\xi(\eta)\big)
\chi_\eta(x),
\end{equation}
since, for $\xi\ne \eta \in X$, $\chi_\xi(\eta) =
\delta_{\xi,\eta}=0$. Taking the covariant derivative in the equation
above yields
\begin{equation}
\label{covariant_local2global}
\nabla \chi^{loc}_\xi(x) = \nabla \chi_\xi(x) + \sum_{\eta\not\in
\Upsilon_\xi}\chi^{loc}_\xi(\eta) \nabla
\chi_\eta(x) .= \nabla \chi_\xi(x) + \sum_{\eta\not\in
\Upsilon_\xi}\big(\chi^{loc}_\xi(\eta)-\chi_\xi(\eta)\big) \nabla
\chi_\eta(x) .
\end{equation}
From \eqref{covariant_local2global} and \eqref{error_loc_chi}, it easily follows that
\[
|\nabla \chi^{loc}_\xi(x) -\nabla \chi_\xi(x)| \le  \| \chi^{loc}_\xi -\chi_\xi\|_{L^\infty} \sum_{\eta\not\in
\Upsilon_\xi} |\nabla \chi_\eta(x)|
\le Ch^J \sum_{\eta \in X} |\nabla \chi_\eta(x)|.
\]
Applying \eqref{covariant_lebesgue_const} to the rightmost inequality then yields the following result:

\begin{lemma}
\label{error_cov_loc_chi}
If $J>1$, then $\| \nabla\chi^{loc}_\xi - \nabla\chi_\xi \|_{L^\infty} \le C\rho_X h_X^{J-1}$.
\end{lemma}

The result we are aiming at is estimating the error made in replacing the \emph{exact} Lagrange functions by the \emph{local} Lagrange functions in computing $A^Y_{\xi,\eta}$. Specifically, define
\begin{equation}
\label{discretized_local_entry}
A^Y_{loc,\xi,\eta}:= \sum_{\zeta \in Y} \big(a\nabla\chi^{loc}_\xi\cdot \nabla \chi^{loc}_\eta+b\chi^{loc}_\xi\chi^{loc}_\eta \big)\big|_\zeta \,w_\zeta.
\end{equation}
We want to estimate $|A^Y_{\xi,\eta}-A^Y_{loc,\xi,\eta}|$. This we do in the proposition below.

\begin{proposition}
Let $A^Y_{\xi,\eta}$ be given by \eqref{discretized_entry} and let $A^Y_{loc,\xi,\eta}$ be as above. Then, for $h_X$ sufficiently small and $J>2$,
\begin{equation}\label{quad_error_exact_lag_loc_lag}
|A^Y_{\xi,\eta}-A^Y_{loc,\xi,\eta}| \le Ch_X^{J-2}
\end{equation}
holds uniformly for $\xi,\eta\in X$ and the set $Y$.
\end{proposition} 

\begin{proof}
Note that, at $x\in \sph^2$, we have
\[
|\nabla\chi^{loc}_\xi\cdot \nabla \chi^{loc}_\eta - \nabla\chi_\xi\cdot \nabla \chi_\eta | \le |\nabla\chi^{loc}_\xi - \nabla\chi_\xi | \, |\nabla \chi_\eta | + |\nabla\chi^{loc}_\eta - \nabla\chi_\eta | \, |\nabla \chi_\xi | + |\nabla\chi^{loc}_\xi - \nabla\chi_\xi | \,|\nabla\chi^{loc}_\eta - \nabla\chi_\eta |.
\]
By this inequality and Theorem~\ref{grad_estimates}, we see that
\[
|\nabla\chi^{loc}_\xi\cdot \nabla \chi^{loc}_\eta - \nabla\chi_\xi\cdot \nabla \chi_\eta | \le C_1h_X^{J-1}q_X^{-1} + C_2h_X^{2J-2} = C_1\rho_X h_X^{J-2} + C_2h_X^{2J-2} \sim C\rho_X h_X^{J-2}.
\]
A similar calculation yields $|\chi^{loc}_\xi \chi^{loc}_\eta - \chi_\xi \chi_\eta | \le Ch_X^J$. From this, the previous inequality, and $h_X$ being small, we have that $|A^Y_{\xi,\eta}-A^Y_{loc,\xi,\eta}|\le Ch_X^{J-2}\sum_{\zeta \in Y}w_\zeta= 4\pi Ch_X^{J-2} \sim Ch_X^{J-2}$. 
\end{proof}

\paragraph{\it Distance estimates.} We have already dealt with a bound on $\|A^Y - \widetilde A^Y \|_2$ in Proposition~\ref{sparse_approx_prop}. We are really only interested in the ``chopped'' version of $A^Y_{loc}$ -- \emph{i.e.},  $\widetilde A^Y_{loc}$, which is defined analogously to $\widetilde A^Y$ in \eqref{truncated_A}. 

\begin{proposition} \label{chopped_discrete2chopped_prop}
For $h_X$ sufficiently small and $J>2$, we have that
\begin{equation}\label{chopped_discrete2chopped_loc}
\| \widetilde A^Y_{loc} - \widetilde A^Y \|_2 <C K^2(\log(h_X))^2h_X^{J-2}.
\end{equation}
\end{proposition}
\begin{proof}
As before, we have $\| \widetilde A^Y_{loc} - \widetilde A^Y \|_2 \le  \max_{\eta\in X} \big(\textstyle{\sum_{\xi\in X}} |\widetilde A^Y_{\xi,\eta} - \widetilde A^Y_{loc, \xi,\eta}|\big)$. For fixed $\eta$, all terms with $d(\xi,\eta) > Kh_x|\log(h_X)|$ are $0$. It follows that 
\[
\textstyle{\sum_{\xi\in X}} |\widetilde A^Y_{\xi,\eta} - \widetilde A^Y_{loc, \xi,\eta}| = \textstyle{\sum_{\xi\in B(\eta,r_h)\cap X}}|A^Y_{\xi,\eta} -A^Y_{loc, \xi,\eta}|,\ r_h=Kh_X|\log(h_X)|.
\]
From \eqref{quad_error_exact_lag_loc_lag}, the difference in the right sum is uniformly bounded by $Ch^{J-2}$. Consequently, applying \eqref{row_cardinality} then yields
\[
\| \widetilde A^Y_{loc} - \widetilde A^Y \|_2 \le \max_{\eta\in X}\big(\textstyle{\sum_{\xi\in X}} |\widetilde A^Y_{\xi,\eta} - \widetilde A^Y_{loc, \xi,\eta}|\big) \le Ch^{J-2}\max_{\eta\in X} \#\{\text{row}\  \eta\} <C K^2(\log(h_X))^2h_X^{J-2},
\]
which is \eqref{chopped_discrete2chopped_loc}.
\end{proof}

\begin{corollary}\label{dist_A_A_chop_loc} Assume that the hypotheses of Theorem~\ref{error_l2_norm_A},   Proposition~\ref{sparse_approx_prop} and Proposition~\ref{chopped_discrete2chopped_prop} hold. Then,
\[
\left\{
\begin{gathered}
\| A-\widetilde A^Y  \|_2 \le  C\big((\log(h_Y))^2 (h_Y/h_X)^{2M} h_X^{-\delta}+|\log(h_X)|h_X^{K\nu-2}\big),\\
\| A-\widetilde A^Y_{loc}  \|_2 \le  C\big((\log(h_Y))^2 (h_Y/h_X)^{2M} h_X^{-\delta}+(\log(h_X))^2h_X^{J-2}\big).
\end{gathered}\right.
\]
\end{corollary}

\begin{proof}
The first result follows from first applying the triangle inequality and the three distance estimates from Theorem~\ref{error_l2_norm_A},   Proposition~\ref{sparse_approx_prop} and Proposition~\ref{chopped_discrete2chopped_prop} to $\| A-\widetilde A^Y \|_2$. Establishing the second is done is a similar way, employing the additional fact that $h_X^{K\nu - 2}|\log(h_X)| <(\log(h_X))^2h_X^{J-2}$, since $K\nu-2>J-2$.
\end{proof}

\paragraph{\it Stability.}The matrices $\widetilde A^Y$ and $\widetilde A^Y_{loc}$ both have roughly the same stability properties as $A$ and $A^Y$. We will establish them in Theorem~\ref{stability_chop_loc_lag} below. To do this, we will need the following elementary result from linear algebra, which we state without proof. 

\begin{lemma}\label{conditioning_approx_A}
Let $S$ and $T$ be Hermitian $n\times n$ matrices and let $S$ be positive definite. If there exists $0\le \varepsilon <1$ such that $\|S-T\|_2 \le \varepsilon \lambda_{\min}(S)$, then $T$ is positive definite,  and, in addition,  these hold: 
\[
\begin{gathered}
(1-\varepsilon)\lambda_{\min}(S) \le \lambda_{min}(T) \le (1+\varepsilon)\lambda_{min}(S), \\
\frac{1-\varepsilon}{1+\varepsilon}\kappa_2(S) \le \kappa_2(T)\le \frac{1+\varepsilon}{1-\varepsilon}\kappa_2(S).
\end{gathered}
\]
\end{lemma}
 
\begin{theorem}\label{stability_chop_loc_lag} 
Suppose that the conditions of Theorem~\ref{norm_inv_A_Y}  and Corollary~\ref{conditioning_A_Y} are satisfied. If $J = K\nu - 4m + 2 - 2\mu >4$, then both $\widetilde A^Y$ and $\widetilde A^Y_{loc}$ are positive definite, have $\lambda_{\min}(\widetilde A^Y)\sim \lambda_{min}(\widetilde A^Y_{loc})\sim q_X^2$, and also have $\kappa_2(\widetilde A^Y)\sim \kappa_2(\widetilde A^Y_{loc})\sim q_X^{-2}$.
\end{theorem}

\begin{proof} Note that $J>4$ implies that $K\nu > 4$, so that, as long as $q_X$ is small, \eqref{sparse_approx_est} holds, and so, using $h_X=\rho_Xq_X$ and applying Theorem~\ref{norm_inv_A_Y}, we have 
\[
\| A^Y - \widetilde A^Y\|_2 \le Cq_X^{K\nu - 4}|\log(q_X)|q_X^2 \le \underbrace{Cq_X^{K\nu - 4}|\log(q_X)|}_{\varepsilon_1} \lambda_{min}(A^Y) = \varepsilon_1 \lambda_{min}(A^Y).
\]
Since $K\nu - 4>0$, we may choose $q_X$ so small that $\varepsilon_1<1$. Lemma~\ref{conditioning_approx_A} then implies the results stated for $\widetilde A^Y$. Using the this result and \eqref{chopped_discrete2chopped_loc}, we get
\[
\| \widetilde A^Y - \widetilde A^Y_{loc}  \|_2 \le \underbrace{Cq_X^{J-4}(\log(q_X))^2}_{\varepsilon_2} \lambda_{\min}(\widetilde A^Y) = \varepsilon_2 \lambda_{\min}(\widehat A^Y),
\]
Because $J>4$, we may choose $q_X$ sufficiently small so that $\varepsilon_2<1$. Applying Lemma~\ref{conditioning_approx_A} then yields the result for $\widetilde A^Y_{loc}$.
\end{proof}

\paragraph{\it Sparse and sparse local Lagrange Galerkin error estimates.} We conclude by giving errors for the $L^2$-Galerkin approximations to $u$, $\tilde u^Y_h$ and $\tilde u^Y_{loc,h}$, which are obtained by discretizing with the chopped matrices $\widetilde A^Y$ and $\widetilde A^Y_{loc}$. The estimates below are gotten in very nearly the same way as the one in Theorem~\ref{discrete_galerkin_error}. The only change is that $ \|(A^Y)^{-1}\| \|A^Y-A\|=\lambda_{min}(A_Y)^{-1}\|A^Y-A\|$ gets replaced by $\lambda_{\min}(\widehat A^Y)^{-1}\| A - \widetilde A^Y\|_2\sim q_X^{-2}\| A - \widetilde A^Y\|_2$ in the first instance, and by $\lambda_{\min}(\widehat A^Y_{loc})^{-1} \| A - \widetilde A^Y_{loc}\|_2\sim q_X^{-2}\| A - \widetilde A^Y_{loc}\|_2$ in the second. 

\begin{theorem}\label{discrete_galerkin_error_sparse_loc}
If the conditions of Theorem~\ref{norm_inv_A_Y}  and Corollary~\ref{conditioning_A_Y} are satisfied, then
\begin{align}
\label{discrete_galerkin_bnd_sparse}
\|u-\tilde u^Y_h\|_{L^2}\le & C\bigg(h_X^{s+2}+\frac{(\log(h_Y) h_Y^{M})^2}{h_X^{2M+2+\delta}} +|\log(h_X)|h_X^{K\nu-4}\bigg) \|f\|_{H_s}, \ K\nu>4,\\
\|u-\tilde u^Y_{loc,h}\|_{L^2}\le  & C\bigg(h_X^{s+2}+\frac{(\log(h_Y) h_Y^{M})^2}{h_X^{2M+2+\delta}} + (\log(h_X))^2h_X^{J-4}\bigg) \|f\|_{H_s}, \ J>4.\label{discrete_galerkin_bnd_sparse_loc}
\end{align}
\end{theorem}

\section{Implementation and Numerical Experiments}
\label{numerics}
This section discusses the practical aspects of implementation and present numerical experiments. We discuss the construction of point sets for the discrete approximation space and quadrature points, the assembly of the stiffness matrix, and the assembly of the right hand side vector. 

The numerical scheme requires two sets of points on the sphere. The coarse set $X$ is used to build a basis for the approximation space $V_{\phi_m,X}$ discussed in section~\ref{small_footprint_bases}. This space is spanned by the Lagrange functions $\{\chi_\xi\}_{\xi\in X}$ defined in \eqref{chi_xi_expan}, which have this form:
\[
\begin{aligned}
\chi_\xi(x) = \sum_{\eta \in X} \alpha_{\eta,\xi} \phi_m(x\cdot \eta) + \sum_{\ell=0}^{m-1} \sum_{k=1}^{2\ell+1}\beta_{l,k,\xi} Y_{\ell,k}(x), \\
\sum_{\eta} \alpha_{\eta,\xi}Y_{\ell,k}(\eta) = 0, \quad 0\le \ell \le m-1, \ 1\le k  \le 2\ell+1.
\end{aligned}
\] 
The $\alpha_{\xi,\eta}$ and $\beta_{\ell,k,\xi}$ coefficients must be computed for each $\xi \in X$. They are determined by $\chi_\xi(\zeta)=\delta_{\xi,\zeta}$ and the second set of equations above. Solving for them is not very difficult, even though each Lagrange function requires all of the points in $X$ for its computation. However, for large data sets, there is a very efficient, parallelizable way to numerically compute local Lagrange functions. These functions, which were introduced in \cite{FHNWW2012}, require relatively few nodes from $X$ and give very good approximations for the $\chi_\xi$'s.

The second, finer set $Y$ is used, in the quadrature formulas, to discretize entries in the stiffness matrix \eqref{discretized_entry}. By Theorem~\ref{error_l2_norm_A}, the mesh norm of the set $Y$ determines the error in the discrete stiffness matrix and should be chosen appropriately to obtain a desired accuracy in numerically approximating it.   The quadrature weights $\set{w_\zeta}_{\zeta \in Y}$ satisfy
\[
\begin{aligned}
\int_{\mathbb{S}^2} \tilde{\chi}_\zeta(x) \, d\mu(x) = w_\zeta,
\end{aligned}
\]
where $\tilde{\chi}_\zeta$ is the Lagrange function centered at $\zeta \in Y$ for the kernel $\phi_m$. The kernel $\phi_m$ need not be the same kernel as the one used in the construction of the approximation space $V_X$. The weights can be computed efficiently by solving a single linear system that can be preconditioned by the local Lagrange functions; see \cite{FHNWW2013} for details. Solving the system with Generalized Minimum Residual method (GMRES) and the local Lagrange preconditioner requires few iterations. Experiments performed in \cite{FHNWW2013} demonstrate that the number of iterations required seems to be independent of the number of points in $Y$.   


The stiffness matrix assembly requires computing the quadrature nodes $Y$ and quadrature weights $\set{w_\zeta}_{\zeta \in Y}$ and the coefficients $\set{\alpha_{\eta,\xi}}_{\xi,\eta \in X}$ and $\set{\beta_{\ell,\xi}}_{\xi \in X}$.  We recall the discrete stiffness matrix entries found via quadrature: 
$$
A^Y_{\xi,\eta} = \sum_{\zeta \in Y} \big(a \nabla \chi_\xi \cdot \nabla \chi_\eta  + b \chi_\xi \chi_\eta\big)\big|_\zeta w_\zeta. 
$$
We provide some details of the computation of $\nabla \chi_\xi \cdot \nabla \chi_\eta(\zeta)$. We expand the Lagrange functions in terms of the surface splines $\phi_m$ (denoted $\phi$) as $\chi_\xi(\zeta) = \sum_{\tau \in X} \alpha_{\tau,\xi} \phi(\zeta,\tau) + p_\xi(\zeta)$ and $\chi_\eta(\zeta) = \sum_{\gamma \in X} \alpha_{\gamma,\eta} \phi(\zeta,\gamma) + p_\eta(\zeta)$.   Let $x = \sin(\theta) \cos(\varphi), y = \sin(\theta) \sin(\varphi), z = \cos(\theta)$ where $0 \leq \theta \leq \pi$ and $0 \leq \varphi \leq 2\pi$.  On $\mathbb{S}^2$, $\nabla f = \frac{\partial f}{\partial \theta} \hat{\theta} + \frac{1}{\sin(\theta)} \frac{\partial f}{\partial \varphi} \hat{\varphi}$. Let $\phi'(\zeta,\tau) = (m-1)(1-\zeta \cdot \tau)^{m-2} \log(1-\zeta \cdot \tau) + \zeta \cdot \tau - 1$. Let $\tau = (\tau_x,\tau_y,\tau_z)$ in Cartesian coordinates. We note that $\frac{\partial \phi(\cdot,\tau)}{\partial x}|_\zeta = \phi'(\zeta,\tau) \tau_x$, and 
similarly for the $y$ and $z$ partial derivatives.  Evaluating the covariant derivative of the restricted surface spline in  Cartesian coordinates then yields
\begin{align*}
\nabla \phi(\cdot,\tau)|_\zeta &= \phi'(\zeta,\tau) \bigg((1-\zeta_x^2)\tau_x - \zeta_x \zeta_y \tau_y - \zeta_x \zeta_z \tau_z)\hat{\imath} \\
&+ (-\zeta_x \zeta_y \tau_x +(1-\zeta_y^2)\tau_y -\zeta_x \zeta_y \tau_z)\hat{\jmath} \\
&+ (-\zeta_x \zeta_z\tau_x -\zeta_y \zeta_z \tau_z + (1-\zeta_z^2) \tau_z)\hat{k} \bigg).
\end{align*}
The evaluation of $\nabla \chi_\xi \cdot \nabla \chi_\eta(\zeta)$ then reduces to
\begin{align}
\nabla \chi_\xi(\zeta) \cdot \nabla \chi_\eta(\zeta) = \big( \textstyle{\sum_{\tau}} \alpha_{\tau,\xi} \nabla \phi(\zeta,\tau) + \nabla p_\xi(\zeta)\big)\cdot \big(\textstyle{\sum_{\gamma}} \alpha_{\gamma,\eta} \nabla \phi(\zeta,\gamma) + \nabla p_\eta(\zeta)\big).
\end{align}
 
\subsection{Numerical Experiments}\label{numerical_experiments}
In this section, we discuss numerical results of various experiments that explore the computational properties of the Galerkin method. We consider different differential operators, explore the effects of the quadrature node density on the $L^2$ error of the discrete solution, and compute condition numbers for the discrete stiffness matrix. We also demonstrate that local Lagrange functions, as discussed in \cite{FHNWW2012}, provide a computationally less expensive approximation space and yield comparable error and condition numbers as the approximation space generated by the Lagrange functions. We choose the spherical basis function $\phi_3(t) = (1-t)^2 \log(1-t)$ to construct the approximation space and $\phi_2(t) = (1-t) \log(1-t)$ for the quadrature weights. We use the minimum energy points for the centers $X$ used in the approximation space $V_{\phi_3,X}$. For the quadrature nodes, we use the icosahedral nodes and quasi-minimum energy points. These points are available for download; see \cite{WrightQuadWeights}. For each experiment, the $L^2$ error is computed by evaluating the discrete solution on a set of evaluation points $E$ and applying the Lagrange function quadrature rule. The set $E$ is 62500 quasi-minimum energy points, which is used for each experiment independent of $X$ and $Y$. Let $N_X$ and $N_Y$ denote the number of points in $X$ and $Y$ respectively.  We approximate $h_Y$ by $\frac{1}{\sqrt{N_Y}}$. 

We first consider the problem $-\Delta u + u = f$ with $u = \exp(\cos(\theta))$ and $f = \exp(\cos(\theta))(\cos^2(\theta) + 2z\cos(\theta))$. In the second and third columns of Table \ref{table:Laplacian} we display the relative $L^2$ errors of the discrete solution for two separate experiments.  To obtain the discrete stiffness matrix, we first fixed $961$ centers for $X$ and varied the number of quadrature points used in $Y$  The quadrature points are icosahedral nodes with between $2,562$ points to $92,162$ points. We theoretically expect the $L^2$ error to be $\calo(|\log(h_Y)|^2 h_Y^4)$. In fact, the numerically observed error is  $\calo(|\log(h_Y)|^2 h_Y^{5.2})$. The experiment was repeated with $N_X = 3721$ minimum energy nodes and using the same $Y$. This time, ignoring the $N_Y=2562$ outlier, $|\log(h_Y)| h_Y^{5.5}$ is observed, indicating that improvement in the theoretical errors rates is possible. The Lagrange basis was used for these two sets.

\begin{table}[b]
\center
\begin{tabular}{||c|c|c||c|c|c||}
\hline\hline
 \multicolumn{2}{||c} \text{ $-\Delta u+u=f $}& & \multicolumn{2}{c} \text{$- \mathrm{div}(  \bfa \cdot  \! \nabla u) + u = f $}
 & \\ 
\multicolumn{2}{||c} \text{Lagrange Basis} & &\multicolumn{2}{c}\text{Local Lagrange Basis}&\\ \hline\hline
$N_Y$& $N_X= 961$ & $N_X=3721$ & $N_Y$& $N_X= 961$ & $N_X=3721$ \\ \hline
  2562    & 7.86e-5     &  2.19e-2 & 2500  & 8.00e-5 & 2.10e-2 \\
   10242 & 2.22e-6    & 3.76e-5    &10000 & 2.46e-6 & 3.23e-5\\
  23042 & 3.34e-7      & 3.83e-6&  22500& 3.02e-7  & 4.78e-6\\
  40962  & 8.96e-8     & 9.32e-7 &40000  & 7.80e-8 & 1.04e-6\\
  92162 & 1.50e-8     & 1.27e-7 &90000 & 1.10e-8 & 1.49e-7\\ \hline\hline
\end{tabular}
\vspace{10pt}
\caption{Both $-\Delta u + u = f$ and $-\text{div}(\bfa\cdot\nabla u)+u=f$ were numerically solved using minimum energy point sets for $X$ and icosahedral point sets for $Y$. The $L^2$ error for all cases was $\calo(|\log(h_Y)| ^2h_Y^{5+})$. Here, $h_Y=N_Y^{-1/2}$. For the first equation, a Lagrange basis was used, and, for the second, a \emph{local} Lagrange basis. }
\label{table:Laplacian}
\end{table}

\begin{figure}[t]
\centering
\begin{tabular}{cc}
\includegraphics[width=0.48\textwidth]{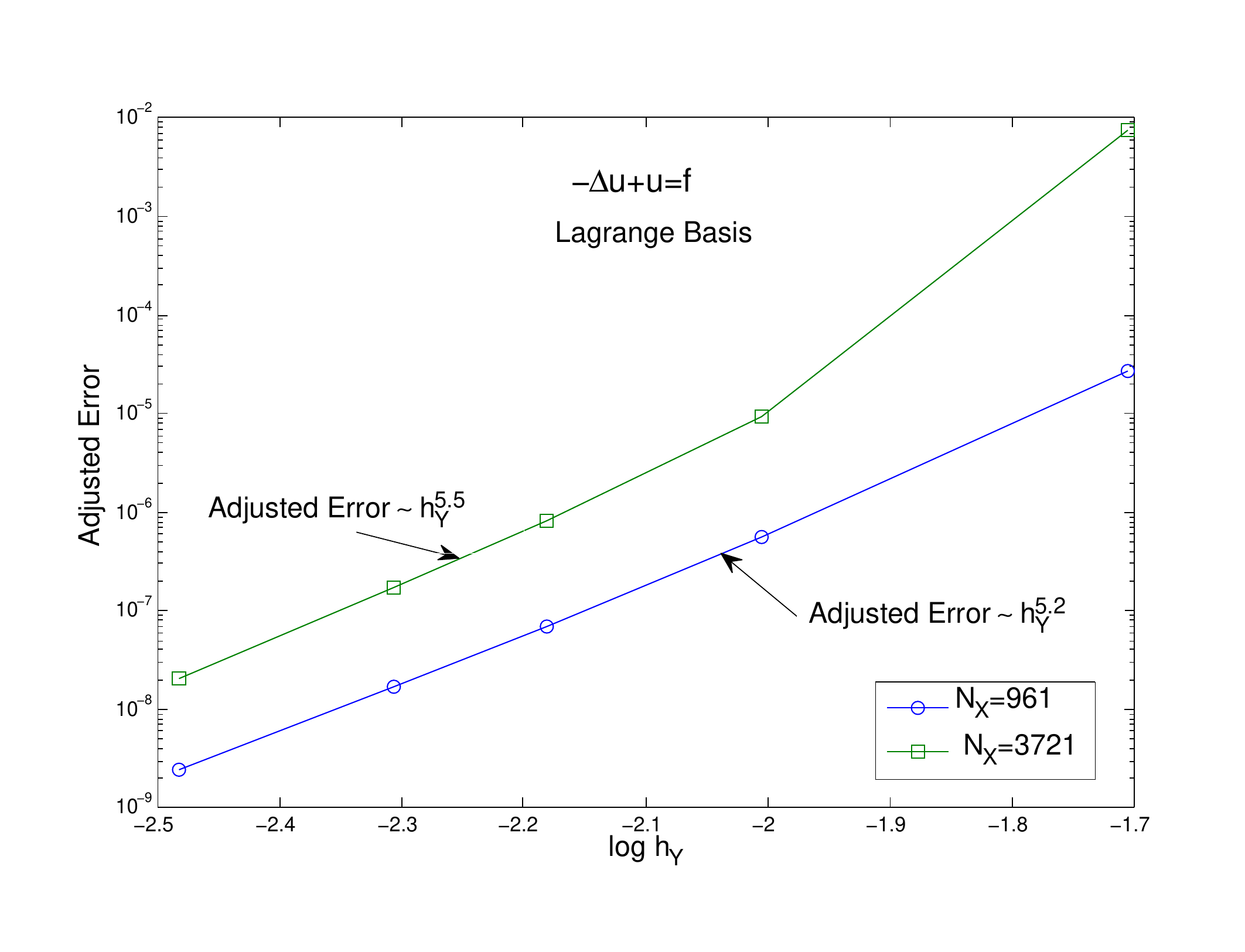} &
\includegraphics[width=0.48\textwidth]{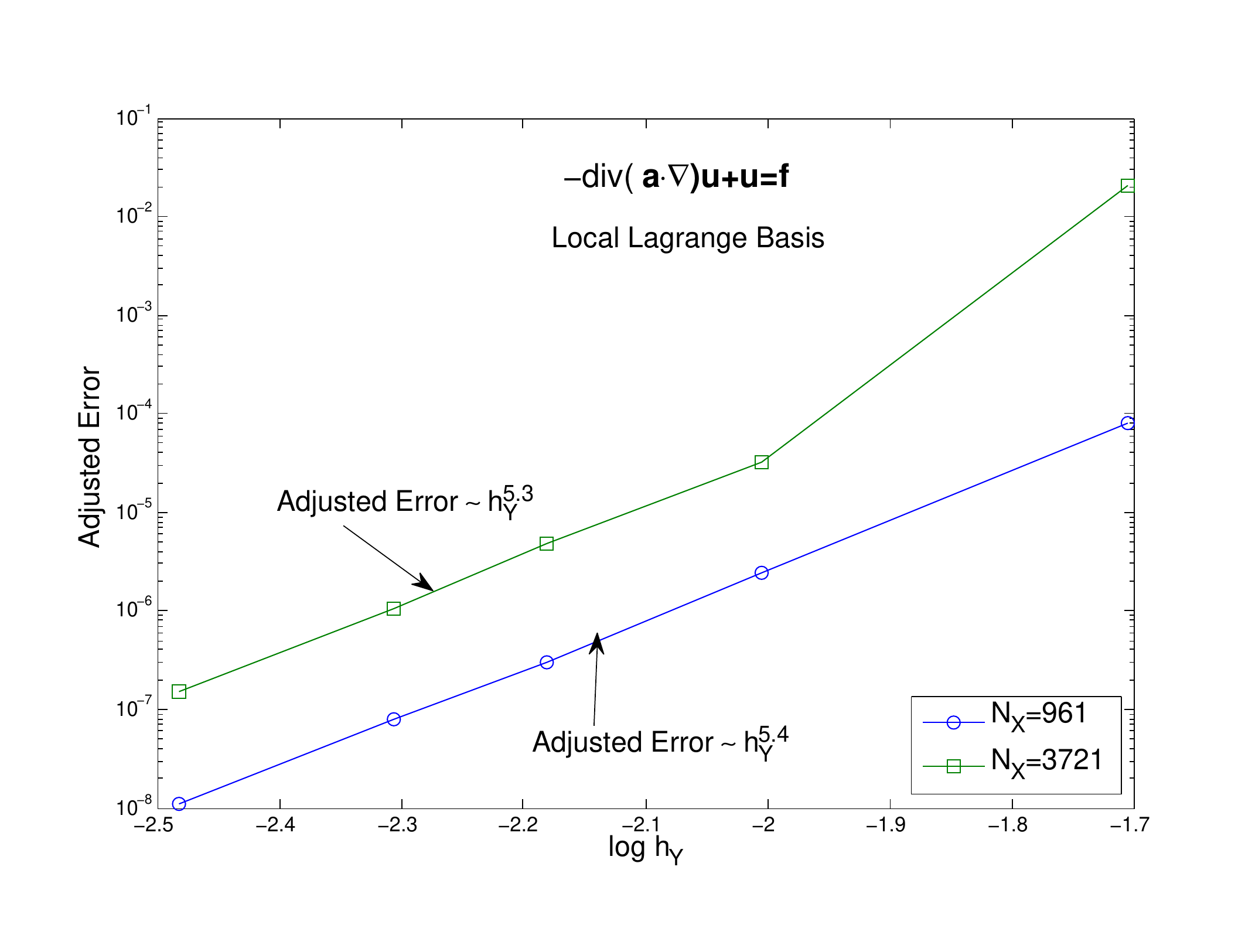} \\
(a) Adjusted $L^2$ error for $-\Delta u+u=f$. & (b) Adjusted $L^2$ error for $-\text{div}(\bfa\cdot \nabla u)+u=f$. \\
\includegraphics[width=0.545\textwidth]{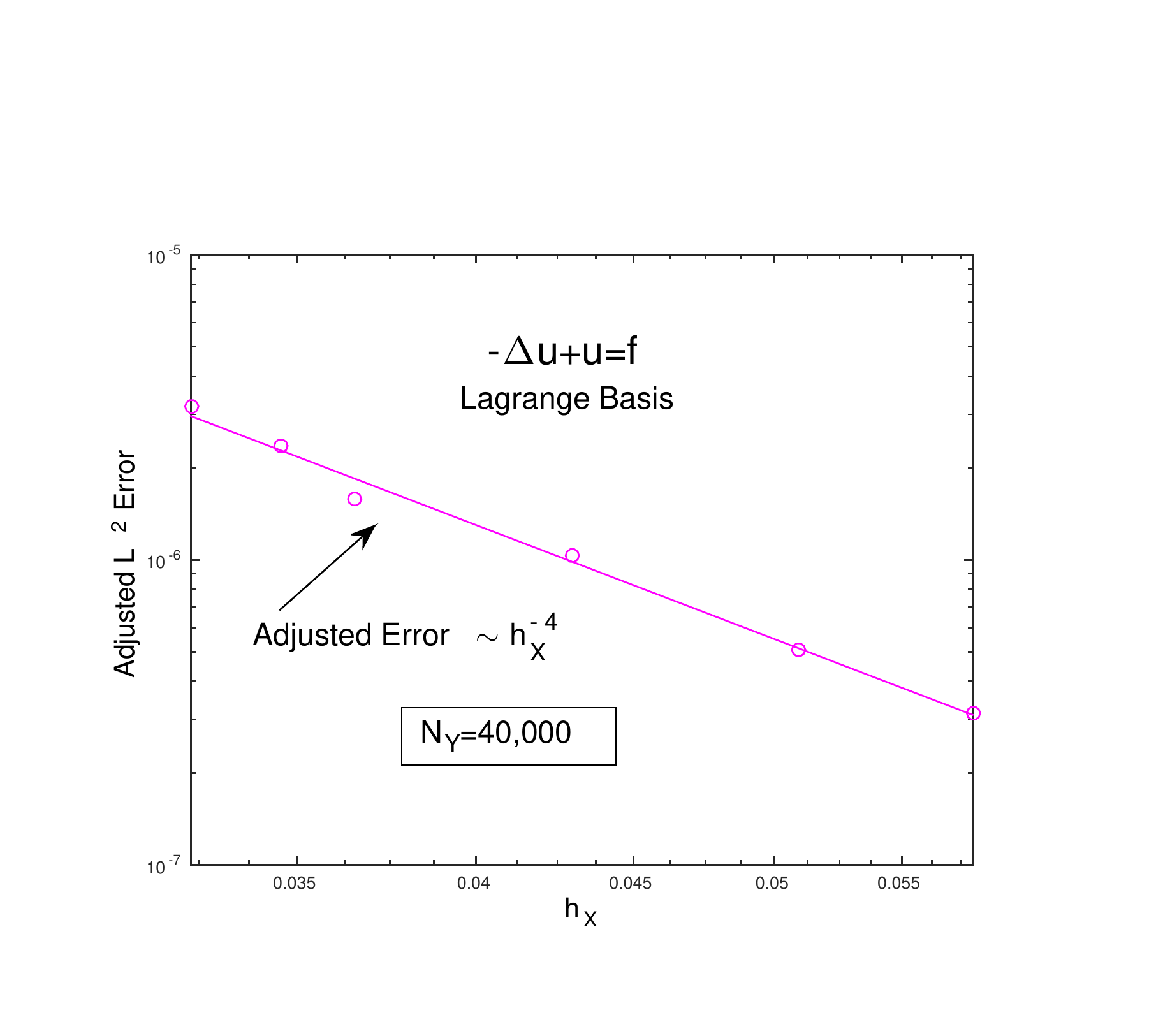} & \includegraphics[width=0.455\textwidth]{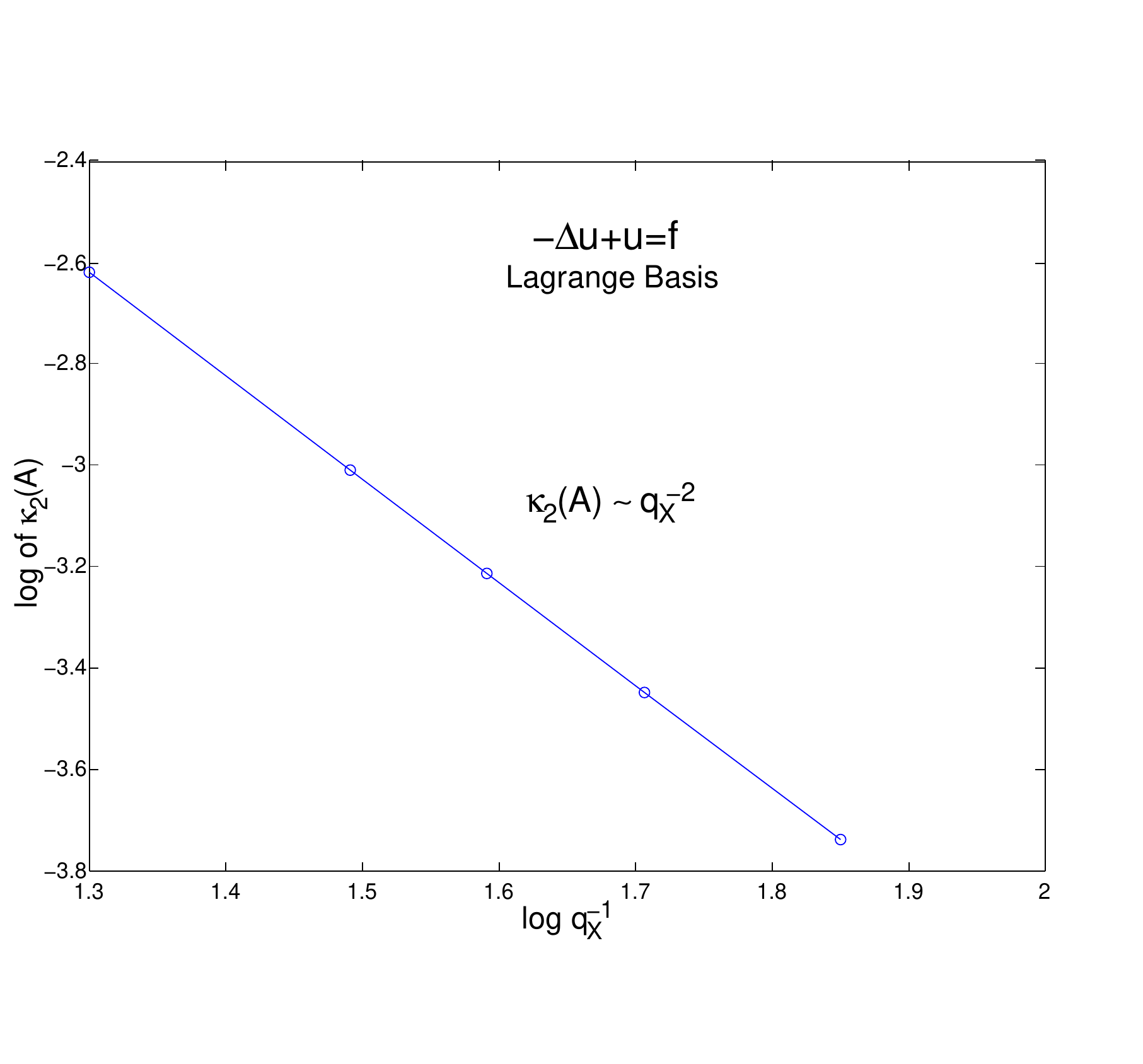}  \\
(c) Adjusted $L^2$ error for $Y$ fixed. &(d) Condition number vs. $q_X^{-1}$. 
\end{tabular}
\caption{In (a) and (b), semi-log plots of the errors (adjusted by removing log factors) for $-\Delta u+u=f$ and  $-\text{div}(\bfa\cdot \nabla u)+u=f$ are shown. The minimum energy points were used for $X$ and icosahedral points were used for $Y$. In (c), a loglog plot of the $L^2$ error vs. $h_X$ is plotted. For this experiment, the number of quadrature points is fixed and the number of centers used for the approximation space varies. In (d), the log of the condition number for the stiffness matrix for $-\Delta u+u=f$ is plotted. }  \label{fig:experiments}
\end{figure}

Next, we treated the problem $ -\text{\rm div}( \mathbf a \! \cdot \! \! \nabla u) + u = f$ for the case in which $\mathbf a = a(\theta,\phi)\mathbf g$, where $\mathbf g$ is the metric tensor for $\sph^2$ and $a(\theta,\phi) = 1-\frac{1}{2} \cos(\theta)$.  We again chose $u = \exp(\cos(\theta))$, which results in the right hand side being 
$
f = \big(-\frac{1}{2}(\cos^3(\theta) +\cos^2(\theta)-5\cos(\theta) +1)+1\big)\exp(\cos(\theta)).
$

We also consider the possibility of using a local Lagrange basis to discretize the PDE. In this case, the approximation space is $V_X = \text{span} \set{\chi_\xi^{loc} \colon\xi \in X}$, where the $\chi_\xi^{loc}$ functions are constructed using only kernels $\phi(\cdot,\eta)$ such that $\text{dist}(\xi,\eta) \leq 7 h_X |\log(h_X)|$. See \cite{FHNWW2012} for a detailed description of the theoretical properties of this basis. The $\chi_\xi^{loc}$'s may be constructed in parallel by solving a small linear system. This reduces computational complexity associated with assembling the $\alpha_{\xi,\eta}$ coefficients. By appropriately tuning the number of kernels used per Lagrange function, the local Lagrange function can be made to satisfy $\|\chi_\xi -\chi_\xi^{loc}\|_{L^{\infty}} \sim h_X^{2m}$, where $m$ is the smoothness of the kernel $\phi$.   For the anisotropic problem, the fifth and sixth columns in Table \ref{table:Laplacian} display the results of the experiment using the local Lagrange bias. For $N_X = 961$, each local Lagrange function is constructed using about $423$ centers   and for $N_X = 3721$, each local Lagrange function is constructed using around $776$ centers, where the number of centers used per kernel is chosen to be all centers with distance at most $7h_X |\log(h_X)|$ from the center. The computed $L^2$ errors from using the local basis versus the full basis are negligible,  confirming the results in section~\ref{Local Lagrange Functions}. Since the local bases offer comparable $L^2$ error while being computationally simpler, they offer no drawbacks when compared to the full basis and certainly are a good choice for the doing the discretization step. The results of the two experiments are plotted in Figure \ref{fig:experiments}(a) and Figure \ref{fig:experiments}(b).

A third experiment was conducted keeping $Y$ with fixed and varying $X$. The result is displayed in Figure~\ref{fig:experiments}(c). In this experiment, the error \emph{increases} with decreasing $h_X$. This is counterintuitive, but in complete agreement with the theory. What this illustrates is that the dominant term in the $L^2$ error comes from quadrature. This is no surprise and is a well-known phenomenon in Galerkin methods.

The condition number of the discrete stiffness matrix is dependent primarily on the separation radius of the centers, $q_X$. We theoretically predicted the condition number to be $\calo(q_X^{-2})$, which we validated numerically. See Figure \ref{fig:experiments}(d). In addition, the theory predicts that changing the \emph{quadrature nodes} should not significantly alter the condition number of the stiffness matrix. Again, this result was validated.

\begin{appendix}

\section{Interpolation Errors and the ``Doubling Trick'' } 
\label{SBF_approximation_power}

In this section we will discuss interpolation errors for spherical basis functions. Previous work on error estimates concentrated on interpolating functions \emph{not} smooth enough to be in the reproducing Hilbert space $\caln$  corresponding the to an SBF $\phi$.  

We will also need error estimates for interpolating functions \emph{smoother} than those in $\caln$. Results of this kind have been developed by Schaback \cite{Schaback-00-1} for positive definite functions on $\RR^n$ and on manifolds. In addition, Fuselier and Wright \cite[Proposition~11]{Fuselier_Wright_2012} give a thorough treatment of the topic. For SBFs, the main result is that if $\caln$ is equivalent to Sobolev space $H_\tau$, $\tau>n/2$, then, for functions in $H_{2\tau}$, the error rate is double the one obtained for functions in $\caln$. This result is known as the ``doubling trick.''

Throughout this section we will assume that an SBF $\phi$ has coefficients $\hat \phi_\ell$ that satisfy the following condition. There are constants $c$, $C$ and $L\in \nats$ such that 
\begin{equation}
\label{phi_ell_bounds}
c(1+\lambda_\ell)^{-\tau} \le \hat \phi_\ell \le C(1+\lambda_\ell)^{-\tau}, 
\end{equation}
holds either for all $\ell\ge 0$ or for all $\ell \ge L+1$. Here $\lambda_\ell=\ell(\ell+n-1)$ is an eigenvalue of~$ -\Delta_{\sph^n}$.

\subsection{Positive definite SBFs}

In this section, we will deal with positive definite SBFs, so $\hat \phi_\ell>0$ for all $\ell$. The proposition below is a  statement of the ``doubling trick'' in the case where $f\in H_{\tau+\alpha}$, $0\le \alpha\le \tau$. We follow this up with a general result combining the doubling trick with estimates from \cite[Theorem~5.5]{Narcowich-etal-07-1}. We separate the two so that the doubling trick itself is clearly stated.

\begin{proposition}
\label{pos_def_case}
Let $\alpha,\beta,\tau\in \RR$, with $\tau>n/2$ and  $\alpha,\beta\in [0,\tau]$. Suppose that $f \in H_{\tau+\alpha}(\sph^n)$ and that (\ref{phi_ell_bounds}) holds for all $\ell\ge 0$. If $h_X$ is sufficiently small, then
\begin{align}
\|f-I_Xf\|_{H_\beta} \leq Ch_X^{\tau+\alpha-\beta} \|f\|_{H_{\tau+\alpha}}.
\label{DoublingTrick}
\end{align}
\end{proposition}

\begin{proof}
We will first deal with the $\beta = \tau$ case. The interpolant $I_Xf$ being the projection of $f$ onto $V_X$ in the native space $\caln$ implies that $\dotp{f-I_X f}{v}_\caln = 0$ for all $v \in V_X$. Consequently,  we have $\dotp{f-I_Xf}{I_Xf}_\caln = 0$ and so  $\|f-I_Xf\|_\caln^2 = \langle f-I_X,f\rangle_\caln $. Let  $g := f-I_Xf$. The previous equation then takes the form $\|g\|_\caln^2 = \langle g,f\rangle_\caln$. From (\ref{native_space_inner_prod}) and the bounds on $\hat \phi_\ell$, we see that
\begin{align}
\|g\|_\caln^2 =\langle g,f\rangle_\caln &= \sum_{l=0}^{\infty} \sum_{m=1}^{d_\ell}(\hat \phi_\ell)^{-1} \hat{f}_{lm} \overline{\hat g_{lm}} \notag \\
 &\le C\sum_{l=0}^{\infty} \sum_{m=1}^{d_\ell}(1+\lambda_\ell)^{(\tau+\alpha)/2} |\hat{f}_{lm}| (1+\lambda_\ell)^{(\tau-\alpha)/2}  |\hat g_{lm}| \notag \\
 &\le C\bigg( \sum_{\ell,m}(1+\lambda_\ell)^{\tau+\alpha} |\hat{f}_{lm}|^2\bigg)^\frac12 \bigg(\sum_{\ell,m}(1+\lambda_\ell)^{\tau-\alpha}  |\hat g_{lm}|^2\bigg)^\frac12 \notag \\
 &= C\|f\|_{H_{\tau+\alpha}} \, \|g\|_{H_{\tau-\alpha}}. \label{NormProduct}
\end{align}
Applying Lemma~\ref{frac_zeros_lemma} to $g=f-I_Xf$ gives us  $\|g\|_{H_{\tau-\alpha}} \le Ch_X^\alpha \|g\|_{H_\tau}$. 
Combining this and \eqref{NormProduct} then yields
\[
\|g\|_\caln^2 \le Ch_X^\alpha \|f\|_{H_{\tau+\alpha}} \, \|g\|_{H_\tau}. 
\]
In addition, the conditions on $\hat \phi_\ell$ imply that $\|g\|_\caln^2 \ge c\|g\|_{H_\tau}^2$, and so
\[
c\|g\|_{H_\tau}^2 \le Ch_X^\alpha \|f\|_{H_{\tau+\alpha}} \, \|g\|_{H_\tau}.
\]
Dividing both sides above by $\|g \|_{H_\tau}$ and replacing $g$ by $g=f-I_Xf$ then yields the $\beta=\tau$ case.  If  $0<\beta \le \tau$, Lemma~\ref{frac_zeros_lemma} implies that $\|f-I_Xf\|_{H_\beta} \leq Ch^{\tau-\beta} \|f-I_Xf\|_{H_\tau}$. Since we have already shown  that $\|f-I_Xf\|_{H_\tau} \le Ch_X^{\alpha} \|f\|_{H_{\tau+\alpha}}$, we have
$
\|f-I_Xf\|_{H_\beta} \leq Ch^{\tau-\beta} \|f-I_Xf\|_{H_\tau} \leq Ch^{\tau+\alpha-\beta} \|f\|_{H_{\tau+\alpha}}.
$
\end{proof}

We want to combine this with the result found in \cite[Theorem~5.5]{Narcowich-etal-07-1}, which deals with estimates for $f\in H_{\mu}$, $\tau\ge \mu>n/2$, to get the following general result.

\begin{theorem}
\label{general_est_SBF}
Let $\beta, \mu, \tau\in \RR$, with $\tau>n/2$, $n/2 < \mu\le 2\tau$, and $\beta \le \min(\mu,\tau)$. Suppose that $f \in H_{\mu}(\sph^n)$ and that (\ref{phi_ell_bounds}) holds for all $\ell\ge 0$. If $h_X$ is sufficiently small, then
\begin{align}
\|f-I_Xf\|_{H_\beta} \leq Ch_X^{\mu-\beta} \|f\|_{H_{\mu}}.
\label{DoublingTrickCombo}
\end{align}
\end{theorem}

\begin{proof}
If $\mu\ge \tau$, the result then follows from \eqref{DoublingTrick}, with $\alpha=\mu - \tau$. If $\mu \le \tau$, then \eqref{DoublingTrickCombo} follows \cite[Theorem~5.5]{Narcowich-etal-07-1}.
\end{proof}

\subsection{Conditionally positive definite SBFs}

The SBFs dealt with above are all strictly positive definite. We will also need to obtain interpolation estimates for the conditionally positive definite SBFs discussed in section~\ref{SBFs}. Recall that for these SBFs,  the $\hat \phi_\ell$'s need only be positive for $\ell >L$. For $0\le \ell \le L$, the $\hat\phi_\ell$'s can be arbitrary.  The interpolation operator  for a conditionally positive definite SBF $\phi$ that reproduces $\Pi_L$ is given \eqref{SPD_SBF_interp_op}. Note that the coefficients $a_{\xi,L}$ and the polynomial $p_{X,L}$ are determined by the requirements that the interpolation condition $I_{X,L}f|_X=f|_X$ hold and also that the coefficients satisfy the condition on the right above. This condition also implies that changing the $\hat\phi_\ell$'s, with $0\le \ell \le L$, will \emph{not} change $I_{X,L}f$, because 
\[
\sum_{\xi\in X}a_{\xi,L}\sum_{\ell=0}^L \hat \phi_\ell Y_{\ell,m}(x)Y_{\ell,m}(\xi) = \sum_{\ell=0}^L \hat \phi_\ell Y_{\ell,m}(x)\underbrace{\sum_{\xi\in X}a_{\xi,L}Y_{\ell,m}(\xi)}_0 =0.
\]
The same reasoning further gives us that the terms $\sum_{\xi\in X}a_{\xi,L}\phi(x\cdot \xi)$ and $p_{X,L}$ are orthogonal. Thus, letting $\proj_{\Pi_L}$ be the orthogonal projection onto $\Pi_L$, we have
\begin{equation}
\label{poly_sum_orthog}
\proj_{\Pi_L} I_{X,L}f=p_{X,L}.
\end{equation}  

These remarks above allow us to assume that $\hat \phi_\ell=1$ for $0\le \ell \le L$, with no loss of generality. We will thus make this assumption. Doing so turns $\phi$ into a strictly positive definite SBF and, consequently, makes possible forming the standard SBF interpolant $I_X\!f(x) = \sum_{\xi\in X}a_\xi\phi(x\cdot \xi)$, with the $a_\xi$'s determined by $I_X\!f|_X=f|_X$. 

The two interpolants $I_X\!f$ and $I_{X,L}f$ are related in several ways. First of all, the difference of the two is given by
\[ 
I_X\! f -I_{X,L}f= \sum_{\xi\in X}\big(a_\xi-a_{\xi,L}\big)\phi((\cdot)\cdot \xi) - p_{X,L}.
\]
Since $(I_X\!f-I_{X,L}f)|_X = 0$, $\sum_{\xi\in X}\big(a_{\xi,L}-a_\xi \big)\phi((\cdot)\cdot \xi)$ interpolates $p_{X,L}$. Or, put another way, $ \sum_{\xi\in X}\big(a_\xi - a_{\xi,L} \big)\phi((\cdot)\cdot \xi) = I_Xp_{X,L}$. Rewriting equation above using this fact yields
\begin{equation}
\label{interp_diff}
I_X\! f -I_{X,L}f= I_Xp_{X,L} - p_{X,L}.
\end{equation}

\begin{theorem}
\label{cond_pos_def_case}
Let $\beta, \mu, \tau\in \RR$, with $\tau>n/2$, $n/2 < \mu\le 2\tau$, and $\beta \le \min(\mu,\tau)$. Suppose that $f \in H_{\mu}(\sph^n)$ and that (\ref{phi_ell_bounds}) holds for all $\ell\ge L+1$. If $h_X$ is sufficiently small, then
\begin{align}
\|f-I_{X,L}f\|_{H_\beta} \leq Ch_X^{\mu-\beta} \|f\|_{H_{\mu}}.
\label{DoublingTrickComboCP_SBF}
\end{align}
\end{theorem}

\begin{proof}
From \eqref{interp_diff} and Theorem~\ref{general_est_SBF} we see that 
\begin{equation}
\label{condit_error_bnd}
\| I_{X,L}f -f \|_{H_\beta}\le \|I_Xf - f \|_{H_\beta} + \|I_Xp_{X,L} - p_{X,L} \|_{H_\beta}\le Ch^{\mu - \beta} \| f \|_{H_\mu} + \|I_Xp_{X,L} - p_{X,L} \|_{H_\beta}.
\end{equation}
Because $p_{X,L}$ is a degree $L$ polynomial, it is analytic, so of course it is in $H_{\mu}$. It follows from Theorem~\ref{general_est_SBF} that  $\|I_X\,p_{X,L} - p_{X,L} \|_{H_\beta} \le Ch_X^{\mu-\beta} \| p_{X,L} \|_{H_{\mu}}$. 
Furthermore, because $p_{X,L}\in \Pi_L$, 
\[
\lambda_L^{-(\mu-\beta)/2} \| p_{X,L}\|_{H_{\mu}} \le \|p_{X,L}\|_{H_\beta} \le \lambda_L^{(\mu-\beta)/2} \| p_{X,L}\|_{H_{\mu}}.
\]
Consequently, $ \|I_Xp_{X,L} - p_{X,L} \|_{H_\beta} \le Ch_X^{\mu - \beta}\|p_{X,L}\|_{H_\beta}$. 
Rewriting \eqref{poly_sum_orthog} as $\proj_{\Pi_L} (I_{X,L}f-f)+\proj_{\Pi_L}f=p_{X,L}$, taking the $H_\beta$ norm, and using $\|f\|_{H_\beta}\le \|f\|_{H_{\mu}}$, we obtain
\[
 \| p_{X,L} \|_{H_\beta} =\|\proj_{\Pi_L} (I_{X,L}f-f)+\proj_{\Pi_L}f\|_{H_\beta} 
\le \| I_{X,L}f-f \|_{H_\beta }+ \|f\|_{H_{\mu}},
\]
from which it follows that
\begin{equation}
\label{poly_error_bnd}
\|I_Xp_{X,L} - p_{X,L} \|_{H_\beta} \le Ch_X^{\mu-\beta} \big(\| I_{X,L}f-f \|_{H_\beta }+ \|f\|_{H_{\mu}}\big).
\end{equation}
From \eqref{condit_error_bnd} and the previous inequality, we have
\[
\| I_{X,L}f -f \|_{H_\beta} \le Ch_X^{\mu-\beta} \|f\|_{H_{\mu}} + Ch_X^{\mu-\beta} \| I_{X,L}f -f \|_{H_\beta}.
\]
Choosing $h_X$ so small that $Ch_X^\alpha <\frac12$ yields 
\[
\| I_{X,L}f -f \|_{H_\tau} \le Ch_X^\alpha \|f\|_{H_{\tau+\alpha}} + \tfrac12 \| I_{X,L}f -f \|_{H_\beta}.
\]
Subtracting $\frac12 \| I_{X,L}f - f \|_{H_{\beta}}$ from both sides and manipulating the result gives us  \eqref{DoublingTrickComboCP_SBF}.
\end{proof}
\section{Sobolev Space Algebras}
\label{Sobolev Space Algebras}

In addition to the estimates on interpolation error estimates derived above, we will need to deal with bounds on Sobolev norms of products of functions. Fortunately,  Coulhon \emph{et al.} \cite{Coulhon-et-al-2001} have established the requisite results. We will state these results for $\sph^n$, in our notation, and only for the ``p=2'' cases.  Before we state these results, we point out that $L^2_\tau=H_\tau$ and that the norm  $\| f \|_{\tau,2}=\| (-\Delta)^{\tau/2}f\|_{L^2} +  \| f\|_{L^2}$, which is defined on \cite[p.~286]{Coulhon-et-al-2001}, is equivalent to $\| f \|_{H_\tau}$. 

\begin{theorem}[{\cite[Theorem 27]{Coulhon-et-al-2001}}] 
\label{leibnitz_rule}
Let $f,g$ be in $H_{\tau}\cap L^\infty$, where $\tau \in [0,\infty)$. Then, $fg\in H_\tau\cap L^\infty$ and there exists $C>0$ such that
\begin{equation}
\label{leibinitz_rule_bnd}
\|fg\|_{H_\tau} \le C\big(\|f\|_{L^\infty} \|g\|_{H_\tau}  + \|g\|_{L^\infty}\|f\|_{H_\tau} \big)
\end{equation}

\begin{proof}
We just need to verify that the conditions in \cite[Theorem 27]{Coulhon-et-al-2001} are satisfied. The parameters in \cite[Theorem 27]{Coulhon-et-al-2001} connect with ours this way: $\tau:=\alpha$, $p_1=q_2=\infty$, $p_2=q_1 =2$.  If we take $\M=\sph^n$, then all the conditions imposed on the manifold, including that of bounded geometry and positive injectivity radius, are satisfied. Thus, \eqref{leibinitz_rule_bnd} holds for $\tau \in [0,\infty)$. 
\end{proof}
\end{theorem}

Using Banach space interpolation methods, Coulhon \emph{et al.}\footnote{As stated, the theorem \cite[Theorem~30]{Coulhon-et-al-2001}, which was  employed in the interpolation process, requires that the manifold be unbounded. However, examining the result as stated in \cite[Theorem~7.4.5]{triebel1992} does not make this assumption, and so the interpolation result holds for compact manifolds as well.} showed that the following holds:

\begin{proposition} [{\cite[p.~334]{Coulhon-et-al-2001}}] \label{grad_soblev_prop}
Let $\tau>0$.  Then $f$ is in $H_{\tau+1}$ if and only if $f$ and $|\nabla f|$ are in $H_\tau$. In addition, $\| f \|_{H_{\tau+1}} \sim \| f \|_{H_\tau} + \| |\nabla f| \|_{H_\tau}$.
\end{proposition}

\begin{corollary}\label{grad_squared_bnd}
Let $f,g\in H_{\tau+1}$ and suppose that $f,g,|\nabla f|, |\nabla g| \in L^\infty$. Then $\nabla f\cdot \nabla g\in H_\tau\cap L^\infty$ and
\[
\| \nabla f \cdot \nabla g \|_{H_\tau} \le C \big(\|f\|_{H_{\tau+1}} + \|g\|_{H_{\tau+1}}\big) \big(\| |\nabla f | \|_{L^\infty} +\| |\nabla g| \|_{L^\infty}\big).
\]

\begin{proof}
We will first prove the result for $f=g$.  By Proposition~\ref{grad_soblev_prop}, $|\nabla f| \in H_\tau$ and $\| |\nabla f| \|_{H_\tau}  \le C\big(\| f \|_{H_{\tau+1}} + \|  f \|_{H_\tau}\big)\le C \| f \|_{H_{\tau+1}}$. Moreover, by Theorem~\ref{leibnitz_rule}, $|\nabla f|^2 \in H_\tau$. Also, since $ |\nabla f| \in L^\infty$, $|\nabla f|^2 \in L^\infty$. The remarks above and \eqref{leibinitz_rule_bnd} then imply that
\[
\| |\nabla f|^2\|_{H_\tau} \le C\| |\nabla f| \|_{H_\tau}\| |\nabla f | \|_{L^\infty} \le C \| f \|_{H_{\tau+1}} \| |\nabla f | \|_{L^\infty}.
\]
The result for $f=g$ then follows immediately. For the general case, just use  $| \nabla (f+g) |^2  - | \nabla (f-g) |^2 = 4\nabla f\cdot \nabla g$ and apply the result for the $f=g$ case.
\end{proof}

\end{corollary}

\end{appendix}
\bibliographystyle{plain}
\bibliography{galerkinsph.bbl}
\end{document}